\documentclass[10pt]{amsart}
\usepackage[utf8]{inputenc}
\usepackage{graphicx, cite, mathtools}
\usepackage{amsmath,amsthm,amssymb,tocvsec2,pdflscape,cite}
\usepackage{latexsym}
\usepackage{esint}
\usepackage{mathtools}
\usepackage{color,soul}
\usepackage{enumerate}
\usepackage[T1]{fontenc}
\usepackage[mathscr]{euscript}

\usepackage[colorlinks=true, allcolors=blue, breaklinks=true]{hyperref}

\allowdisplaybreaks
\numberwithin{equation}{section}
\usepackage{cleveref}
\newtheorem{theorem}{Theorem}[section]
\newtheorem{thma}{Theorem}
\newtheorem{lemma}[theorem]{Lemma}
\newtheorem{proposition}[theorem]{Proposition}
\newtheorem{corollary}[theorem]{Corollary}

\theoremstyle{definition}
\newtheorem{definition}{Definition}

\title[Leray-Trudinger Type Exponential Integrability]{Leray-Trudinger Type Exponential Integrability in Log-Weighted Sobolev Spaces }

\author{Adimurthi}
\address{Department of Mathematics, Indian Institute of Technology, Kanpur, India}
\email{adiadimurthi@gmail.com}

\author{Sourav Ghosh}
\address{Department of Mathematics, Indian Institute of Science, Bangalore, India}
\email{souravghosh1@iisc.ac.in}

\author{Arka Mallick}
\address{Department of Mathematics, Indian Institute of Science, Bangalore, India}
\email{arkamallick@iisc.ac.in}

\subjclass[2020]{Primary 46E35; Secondary 46E30, 26D10, 26D15}
\keywords{Moser-Trudinger inequalities, Leray-Trudinger inequalities, weighted Sobolev spaces, logarithmic weights, Leray energy, Hardy inequalities, Orlicz spaces.}

\begin{document}
\begin{abstract}
In this article, we conduct a comprehensive study of weighted Sobolev spaces with logarithmic weights, orginially introduced by Calanchi and Ruf in \cite{RufClanchi2015, RufN}, to analyze the sharp exponential integrability of \emph{radial functions} belonging to these spaces. By exploring the connection between these logarithmically weighted energies and the Leray energy, we expand the framework to incorporate \emph{non-radial} functions. More precisely, we establish optimal exponential integrability for general functions in the spirit of optimal Leray-Trudinger inequalities established in \cite{di2024optimal}. Furthermore, we prove sharp versions of these inequalities when restricted to radial functions. Notably, the inequalities presented here are fundamentally different in nature from those of Calanchi and Ruf, for which the non-radial extension fails to hold.
\end{abstract}

\maketitle
\section{Introduction and Statement of Main Results}
The Moser-Trudinger inequality is one of the most fundamental tools in the theory of PDE. The inequality in its sharpest form, which is  derived by Moser in \cite{Moser} says

\begin{align} \label{moser-trudinger}
        \sup_{\left\lbrace u \in W^{1,N}_{0}(\Omega) : \lVert \nabla u \rVert_N \le 1\right \rbrace}  \fint_{\Omega} e^{\alpha \lvert u \rvert^{\frac{N}{N-1}}} \, dx \le C,
    \end{align}
    if and only if $\alpha \le \alpha_{N}  \coloneqq N \omega_{N-1}^{\frac{1}{N-1}} ,$ for any bounded domain $\Omega \subset \mathbb{R}^N$. Here $\omega_{N-1}$ denotes the surface area of $\partial \mathbb{B}_N$, where $\mathbb{B}_N \coloneqq B(0,1) \subset \mathbb{R}^N$, and $C \equiv C\left(N\right)>0$ is a positive constant which depends only on the dimension $N$. A non-sharp version of \eqref{moser-trudinger} follows from the earlier works of Yudovich \cite{Yudovich},  Peetre \cite{Peetre_Lorentzembedding}, Pohozaev \cite{Poh} and Trudinger \cite{Trudinger}, where it was proved that $ W^{1,N}_{0}(\Omega) \hookrightarrow L^{\psi_N}\left(\Omega\right)$. Here $L^{\psi_N}$ is the Orlicz-Space given by the Young function $\psi_{N}(t)= \exp(t^{N'})-1$, for $t\geq0$. The optimality of this embedding was established in a subsequent paper by Hempel, Morris and Trudinger \cite{HMT}. 
    \par In the intervening years, numerous refinements of \eqref{moser-trudinger} of the following type have been explored 
     \begin{align}\label{reduced energy with gen potential}
     \sup_{\left\lbrace u \in C_{c}^{1}\left(\Omega\right) : E_{V}\left(u\right) \le 1\right \rbrace}  \int_{\Omega} e^{\alpha u ^{2}} \, dx <\infty, \ \alpha\leq \alpha_{2} =4\pi
    \end{align}
  where $E_{V}$ is a nonnegative energy functional defined on $C_c^{\infty}\left(\Omega\right)$ by 
  \begin{align}\label{reduced energy}
  E_{V}\left(u\right):= \lVert \nabla u \rVert^{2}_{2} - \int_{\Omega} V(x) \left\lvert u(x) \right\rvert^2 dx.
  \end{align}
 To derive such a result, it is necessary to assume that the functional satisfies the \emph{weak coercivity} property. More precisely, there must exists an open set $\Omega_1 \subset \subset \Omega$ and a constant $C>0$ such that 
 \begin{align}\label{weak-coercivity}
 E_{V}\left(u\right) \geq C \int_{\Omega_1} \left\lvert u\right \rvert^2 , \mbox{ for all } u \in C_c^{1}\left(\Omega\right).
 \end{align}
   Otherwise, by the ground state alternative of Murata \cite{murata} (see also \cite{pinchover-tintarev}) we get the existence of a sequence $u_{k}\in C_c^{\infty}\left(\Omega\right)$ converging to a nontrivial function $\phi$ in $H^{1}_{loc}\left(\Omega\right)$ which makes the inequality \eqref{reduced energy with gen potential} invalid.   
 If we assume that the potential is a constant, specifically if  $V\equiv \lambda < \lambda_{1} \left(\Omega\right)$ on $\Omega$, where $\lambda_{1} \left(\Omega\right)$ is the first eigenvalue of the Dirichlet Laplacian, then \eqref{reduced energy with gen potential} directly follows form the seminal work of  Adimurthi and Druet \cite{Adi-Druet} (see \cite{Yang-Adi-Druet-N} for a higher dimensional generalization). Their work was pioneering and introduced  a novel blow up analysis in dimension $N=2$. Analogous blow up analysis were employed in a non-compact setting by Dong and Ye in \cite{WangYe}. In fact they established \eqref{reduced energy with gen potential} with $\Omega =\mathbb{B}_{2}$, where $V(x)= \left(2-2\left\lvert x \right\rvert^{2}\right)^{-2}$, for $x\in \mathbb{B}_{2}$ (see  \cite{LY} and \cite{vanNguyen-N-Hardy-MT} for generalizations). Note that the functional
 $$E_{V}\left(u\right)=  \lVert \nabla u \rVert^{2}_{2} - \frac{1}{4}\int_{\Omega} \frac{1}{\left(1-\left\lvert x \right\rvert^{2}\right)^{2}} \left\lvert u(x) \right\rvert^2 dx$$ is weakly coercive because of the improved hardy inequalities derived in \cite{BreMar}. Finally, in \cite{T1}, Tintarev considered a general radial potential $V$ in $\mathbb{B}_2$ for which \eqref{reduced energy} is weakly coercive. He established \eqref{reduced energy with gen potential} under the assumption that there exists $\kappa>0$ such that, 
 \begin{align}\label{tintarev condition}
 \lim_{r\to 0} r^2 \left(\log \frac{1}{r}\right)^{2+\kappa}V(r)=0
 \end{align}
 extending the results of both \cite{Adi-Druet} and \cite{WangYe}.
 
 In the borderline case $\kappa= 0$, specifically for the potential $V(x)= V_{Leray}(x) = \left(2\lvert x  \rvert \log \left(1/|x|\right) \right)^{-2}$,  the nonnegativity of the functional in \eqref{reduced energy} was proved by Leray in \cite[Inequality (5), Chapter III]{leray}. More generally, for $N\geq 2$ and a bounded domain $\Omega \subset \mathbb{R}^N$ containing the origin, with $R_{\Omega} \coloneqq \sup_{\Omega}\lvert x \rvert$ and $X_1(r): = \left(\log (e/r)\right)^{-1},$ for $0\leq r\leq1$, the \emph{Leray energy} functional
\begin{align}\label{reduced energy leray}
    I_{N,\Omega}[u] \coloneqq \int_{\Omega} \lvert \nabla u \rvert^N \, dx - \left(\frac{N-1}{N}\right)^N \int_{\Omega} \frac{\lvert u \rvert^N}{\lvert x \rvert^N} X_1^N\left(\frac{\lvert x \rvert}{R_{\Omega}}\right) \, dx, \  u \in W^{1,N}_0(\Omega)
\end{align}
is weakly coercive due to the improved hardy inequalities derived in \cite{ACR} and \cite{BFT} independently. However, Psaradakis and Spector \cite{PsaradakisJFA} observed that \eqref{reduced energy with gen potential} (for any $\alpha>0$), as well as its higher dimensional analogues, fails to hold, thus remarkably giving an example of weakly coercive functional for which \eqref{reduced energy with gen potential} is false. In the same paper, they proved that for any $\varepsilon>0$, there exists constant $\alpha \equiv \alpha\left(N, \varepsilon\right)>0$ such that
\begin{align}\label{leray-trudinger}
 \sup_{\left \lbrace u \in W^{1,N}_{0}\left(\Omega\right) : I_{N,\Omega}\left[u\right] \le 1\right\rbrace} \int_{\Omega} e^{\alpha\left(\lvert u \rvert X_1^{\varepsilon}\left(\frac{\lvert x \rvert}{R_{\Omega}}\right)\right)^{N'}} \, dx < \infty,
\end{align}
 which was termed the Leray-Trudinger inequality. 
 \par In a subsequent article by Tintarev and the third author \cite{TinMal}, it was observed that the growth function within the integrand of \eqref{leray-trudinger} is suboptimal. Indeed, because of the ground state representation (see for example, \cite[Proposition 2.6]{PsaradakisJFA})
 \begin{align}\label{ground state rep}
    I_{2,\Omega}[u] = \int_{\Omega} \lvert \nabla v \rvert^2 X_1^{-1} \, dx \coloneqq J_{2,\Omega} [v] \mbox{ for } u\in C_c^1\left(\Omega\setminus \lbrace 0\rbrace\right) \mbox{ with } v= X_1^{\frac{1}{2}} u,
\end{align}
for $X_2: = X_1\circ X_1$ and all \emph{radial} functions $u \in W^{1,2}_{0}\left(\mathbb{B}_2\right) $ and $v \in C_c^1\left(\mathbb{B}_{2}\setminus \lbrace 0\rbrace\right)$ we have
\begin{align}\label{dim2 leray trudinger}
\sup_{I_{2,\mathbb{B}_{2}}\left[u\right] \le 1} \int_{\mathbb{B}_{2}} e^{\alpha\left(\lvert u \rvert X_2^{1/2}\left(\lvert x \rvert\right)\right)^{2}} \, dx = \sup_{J_{2,\mathbb{B}_{2}}\left[v\right] \le 1} \int_{\mathbb{B}_{2}} e^{\alpha v^{2} X_2 X^{-1}_1} \, dx <\infty
\end{align}
if $\alpha < 4\pi$. Here the finiteness follows from \cite[Lemma 5]{RufClanchi2015}. This observation led to an intermediate improvement of \eqref{leray-trudinger} in \cite{TinMal} where $X_1^{\varepsilon}$ is replaced with $X_2^{2/N}$. They also proved that \eqref{leray-trudinger} fails if $X_1^{\varepsilon}$ is replaced with $X_2^{\kappa}$ for any $\kappa<1/N$. Finally, the optimal version of \eqref{leray-trudinger}, which says that for any bounded domain $\Omega$ in $\mathbb{R}^N$, containing the origin, we have  
\begin{align}\label{op-leray-trudinger}
 \mathcal{I}_{N,\alpha}:= \sup_{\left \lbrace u \in W^{1,N}_{0}\left(\Omega\right) : I_{N,\Omega}\left[u\right] \le 1\right\rbrace} \fint_{\Omega} e^{\alpha\left(\lvert u \rvert X_2^{\frac{1}{N}}\left(\frac{\lvert x \rvert}{R_{\Omega}}\right)\right)^{N'}} \, dx < \infty,
\end{align}
for some $\alpha \equiv \alpha\left(N\right)>0$, was established by Di Blasio, Pisante and Psaradakis in \cite{di2024optimal}.
 
 \par Despite these developments, establishing the sharp form of \eqref{op-leray-trudinger} remains an open question. Furthermore, \eqref{ground state rep} and \eqref{dim2 leray trudinger} indicate that an investigation into weighted Sobolev spaces is essential. Consequently, this article focuses on the optimal embedding properties of general weighted Sobolev spaces  $W^{1,N}_{0}(\Omega, w_{i\beta})$. Spaces of this type were initially  introduced by Calanchi and Ruf in \cite{RufClanchi2015} and \cite{RufN} as the closure of  $C^1_c(\Omega)$ under the norm 
 \begin{align}\label{weighted norms}
 \left \lVert \nabla u \right\rVert_{N, w_{i\beta}, \Omega} := \left(\int_{\Omega} \left\lvert \nabla u\right\rvert^N w_{i\beta}\ dx \right)^{\frac{1}{N}},
 \end{align}
where $i=1,2$, $\beta >0$, and the weights are given by,

\begin{align} \label{weights w}
   \begin{cases}
    w_{1\beta}(x) &= Y_{1}^{-\beta(N-1)}\left(\frac{\lvert x \rvert}{R_{\Omega}}\right), \mbox{ with }Y_1(r)=  \left(\ln \frac{1}{r} \right)^{-1},\mbox{ and }\\ 
    w_{2\beta}(x) &= X_{1}^{-\beta(N-1)}\left(\frac{\lvert x \rvert}{R_{\Omega}}\right),
    \end{cases}
\end{align}
for $x\in B\left(0, R_{\Omega}\right)$.
If the context is clear, we will write $X_1$, $Y_1$ and $\left \lVert \nabla u \right\rVert_{N, w_{i\beta}} $ to mean $X_1\left({\lvert x \rvert}/{R_{\Omega}}\right)$,  $Y_1\left({\lvert x \rvert}/{R_{\Omega}}\right)$ and $\left \lVert \nabla u \right\rVert_{N, w_{i\beta}, \Omega}$  respectively. We extend $w_{1\beta}$ and $w_{2\beta}$ to all of $\mathbb{R}^N$ by setting them to $1$ on $B^{c}\left(0,R_{\Omega}\right)$. Note that, 
\begin{align}\label{A-N prop}
\begin{cases}
w_{1\beta}\in A_{N}, \mbox{ if } 0<\beta<1 \mbox{ and }\\
w_{2,\beta} \in A_{N}, \mbox{ if } 0<\beta<\infty.
\end{cases}
\end{align}
On the other hand, if $\beta \geq 1$, then $w_{1\beta} \notin A_{N}$. Here, $A_{N}$ denotes the Muckenhoupt’s class of exponent $N$  (see Lemma \ref{ap weights lem} in Section \ref{prel}). Therefore, throughout the article we will avoid considering the weight $w_{1\beta}$, when $\beta\geq 1$.

The embedding properties of $W^{1,N}_{0, rad}\left(\mathbb{B}_{N}, w_{i\beta}\right)$ i.e. the subspace of all radial functions within $W^{1,N}_{0}\left(\mathbb{B}_{N}, w_{i\beta}\right)$ was studied in Calanchi and Ruf \cite{RufClanchi2015, RufN}. In fact they established sharp exponential integrability results in the spirit of \eqref{moser-trudinger} for functions in $W^{1,N}_{0, rad}\left(\mathbb{B}_{N}, w_{i\beta}\right)$. However, aside form the fact that (see \cite[Proposition 8]{RufClanchi2015} and \cite[Proposition 12]{RufN}) 

\begin{align}\label{ruf calanchi sup infinite}
\sup_{\left \lbrace u \in W^{1,N}_{0}\left(\mathbb{B}_N, w_{i\beta}\right) :\left \lVert \nabla u \right\rVert_{N, w_{i\beta}} \le 1\right\rbrace} \fint_{\mathbb{B}_N} e^{\alpha \lvert u \rvert^{N'}} \, dx =\infty
\mbox{ for any } \alpha > N \omega_{N-1}^{\frac{1}{N-1}}, 
\end{align}
 and for $i=1,2, 0<\beta \leq 1$, not much is known about the space $W^{1,N}_{0}\left(\mathbb{B}_{N}, w_{i\beta}\right)$. 
\par {\bf{Main results for $\beta=1$}}: 

In order to understand the embedding properties for the full space, let us consider  the specific case where $N=2$, $i=2$ and $\beta=1$. Since $w_{21}\geq 1$, so by \eqref{moser-trudinger} and \eqref{ruf calanchi sup infinite} we conclude that 
\begin{align}\label{moser-truding w21}
\sup_{\left \lbrace u \in W^{1,N}_{0}\left(\mathbb{B}_N, w_{21}\right) :\left \lVert \nabla u \right\rVert_{N, w_{21}} \le 1\right\rbrace} \fint_{\mathbb{B}_N} e^{\alpha \lvert u \rvert^{N'}} \, dx <\infty
\mbox{ iff } \alpha \leq N \omega_{N-1}^{\frac{1}{N-1}}, 
\end{align}
which implies 
\begin{align}\label{embedding for dimension N}
W^{1,N}_{0}\left(\mathbb{B}_{N}, w_{21}\right) \hookrightarrow L^{\psi_N}\left(\mathbb{B}_{N}\right), \mbox{ with } \psi_N(t)=\exp(t^{N'})-1 , \mbox{ for } t\geq 0. 
\end{align}
Since the energy is weakened by the weight, there is a scope of improvement of this inequality. This was first explored by Calanchi and Ruf in \cite{RufClanchi2015, RufN}, where they proved 
    \begin{align}\label{ruf2}
        \sup_{\left \lbrace u \in W^{1,N}_{0,rad}\left(\mathbb{B}_N,w_{21}\right) : \left \lVert \nabla u\right\rVert_{N, w_{21}}\le 1\right\rbrace} \int_{\mathbb{B}_N}  e^{\alpha e^{\omega_{N-1}^{\frac{1}{N-1}} \lvert u \rvert^{N'}}} \, dx < \infty, \mbox{ iff } \alpha \leq N.
    \end{align}
    However, this does not have any non-radial extension due to \eqref{ruf calanchi sup infinite}. Neverthelss, we notice that, the first equality in \eqref{dim2 leray trudinger} holds true even for non radial functions $u$ and $v$.  By combining this observation with that fact $W^{1,2}_{0}\left(\mathbb{B}_{2}, w_{21}\right)=W^{1,2}_{0}\left(\mathbb{B}_{2}\setminus \lbrace 0 \rbrace, w_{21}\right)$ (see Theorem \ref{approximation} in Section \ref{prop w-sob spaces}) and the optimal Leray-Trudinger inequality \eqref{op-leray-trudinger} we deduce that 
\begin{align}\label{optimal embedding for dimension 2}
W^{1,2}_{0}\left(\mathbb{B}_{2}, w_{21}\right) \hookrightarrow L^{\phi_2}\left(\mathbb{B}_{2}\right), 
\end{align}
where for $N\geq 2$, $L^{\phi_N}\left(\Omega \right)$ denotes the Musielak-Orlicz space (see Definition \ref{M-O space} below) with the generalized $\Phi$-function $\phi_{N}$ defined by 
\begin{align}\label{gen young function}
\phi_{N}\left(x, t\right): = \exp{\left(X_2\left( \frac{\lvert x\rvert}{R_{\Omega}} \right)X_1^{-1}\left( \frac{\lvert x\rvert}{R_{\Omega}} \right) t^{N'}\right)}-1,
\end{align}
for $x\in \Omega \subset \mathbb{R}^N$ and $t\geq 0.$ Clearly, \eqref{optimal embedding for dimension 2} improves \eqref{embedding for dimension N} in dimension $N=2$, as $X_2X_1^{-1}(|x|/R_\Omega) \geq 1$, for in $x\in \Omega$.
Our first result generalizes \eqref{optimal embedding for dimension 2} to any dimension $N>2$. More precisely, we have the following theorem.

\begin{theorem} \label{non radial moser}
    Let $N\geq2$ and $\Omega\subset \mathbb{R}^N$ be any bounded domain containing the origin. Then there exists a positive constant $\alpha \equiv \alpha\left(N\right)$ such that,
    \begin{equation} \label{equation non radial moser}
\mathcal{J}_{N, \alpha, w_{21}}:=  \sup_{\left \lbrace u \in W^{1,N}_{0}\left(\Omega, w_{21}\right) :\left \lVert \nabla u \right\rVert_{N, w_{21}} \le 1\right\rbrace}  \fint_{\Omega} e^{\alpha \lvert u \rvert^{N'} X_2 X_1^{-1}} \, dx <\infty.
    \end{equation}
In particular, $W^{1, N}_{0}\left(\Omega, w_{21}\right) \hookrightarrow L^{\phi_N}\left(\Omega\right)$. Moreover, if $f: \Omega \to [0,\infty]$ is measurable  and there exists $a\in \overline{\Omega}$ such that $fX_2^{-1}X_1 \to \infty$ as $x\to a$, then $W^{1, N}_{0}\left(\Omega, w_{21}\right)$ does not embed in $L^{\phi}\left(\Omega\right)$, where  $\phi(x,t) := e^{t^{N'}f(x)}-1$, for $x\in \Omega$ and $t\ge0$. In other words, 
 \begin{equation} \label{equation non radial moser optimal}
 \sup_{\left \lbrace u \in W^{1,N}_{0}\left(\Omega, w_{21}\right) :\left \lVert \nabla u \right\rVert_{N, w_{21}} \le 1\right\rbrace}  \fint_{\Omega} e^{\alpha \lvert u \rvert^{N'} f} \, dx = \infty, \mbox{ for any } \alpha>0.
    \end{equation}
\end{theorem}
 \par We first observe that \eqref{ruf2} yeilds the embedding of $W^{1, N}_{0, rad}\left(\mathbb{B}_N, w_{21}\right)$ into  $L^{\psi_{N1}}\left(\Omega\right)$, where the generalized Young's function $\psi_{N1}$ is defined by
$$\psi_{N1}(t) = \exp\left(N e^{t^{N'}} - N \right) - 1,\mbox{ for } t\geq 0.$$ Furthermore, Proposition \ref{M-O embedding} in Section \ref{prel} establishes that  $L^{\psi_{N1}}(\mathbb{B}_N)$ embeds into $L^{\phi_N}(\mathbb{B}_N)$. Since these two spaces do not coincide in general, Theorem \ref{non radial moser} does not extend the embedding given by \eqref{ruf2} in the non-radial setting.

\par Next we note that, Proposition $2.6$ in \cite{PsaradakisJFA} implies the following ineqaulity
\begin{equation}\label{leray vs weight moser}
\mathcal{I}_{N,\alpha} \leq \sup_{\left \lbrace u \in W^{1,N}_{0}\left(\Omega, w_{21}\right) :\left \lVert \nabla u \right\rVert_{N, w_{21}} \le 1\right\rbrace}  \fint_{\Omega} e^{\alpha\left(2^{N-1}-1\right)^{-\frac{1}{N-1}} \lvert u \rvert^{N'} X^{\frac{1}{N-1}}_2 X_1^{-1}} \, dx.
\end{equation}
As mentioned previously, equality in \eqref{leray vs weight moser} occurs if $N=2$.  However, for $N>2$, the quantities 
$$I_{N,\Omega}[\cdot] \mbox{ and }\left \lVert \nabla \left( X_{1}^{\frac{1}{N'}} \cdot \right)\right\rVert^{N}_{N, w_{21}, \Omega}$$ 
are not equivalent on the space $C_c^{\infty}(\Omega \setminus \{0\})$ (see \Cref{lerray fail} in \Cref{app A}). Consequently, equality is generally not expected in \eqref{leray vs weight moser}.  Indeed, \Cref{non radial moser}, particularly \eqref{equation non radial moser optimal}, shows that the RHS of \eqref{leray vs weight moser} is infinite for any $\alpha>0$. Thus, for $N>2$, \eqref{equation non radial moser} does not imply \eqref{op-leray-trudinger} and vice versa. 

 Our next result directly improves \eqref{moser-truding w21} in the space $W^{1,N}_{0, rad}\left(\mathbb{B}_{N}, w_{21}\right)$. 
\begin{proposition} \label{moser}
     Let $N\geq 2$. Then 
     \begin{equation} \label{equation sharp radial moser}
\mathcal{J}_{N, \alpha, w_{21},rad}:=  \sup_{\left \lbrace u \in W^{1,N}_{0,rad}\left(\mathbb{B}_{N}, w_{21}\right) :\left \lVert \nabla u \right\rVert_{N, w_{21}} \le 1\right\rbrace}  \fint_{\mathbb{B}_N} e^{\alpha \lvert u \rvert^{N'} X_2 X_1^{-1}} \, dx <\infty,
    \end{equation} iff $\alpha \leq \alpha_{N}$. Moreover, for $\alpha \leq \alpha_{N}$ the supremum is attained in $W^{1,N}_{0,rad}\left(\mathbb{B}_{N}, w_{21}\right)$.
\end{proposition}

In the unweighted case, the corresponding version of \Cref{moser} implies \eqref{moser-trudinger} via the P\'olya-Szeg\"o inequality. However, with the weight $w_{21}$, such an inequality does not hold (see \Cref{pz fails} in \Cref{app B}). Thus the sharp version of \eqref{equation non radial moser} still remains an open question. 

\par We now observe that, \Cref{radial lemma}(iii) of Section \ref{prel} is valid with the weight $Y_2^{-1}(x): = \ln \ln \left(e/\lvert x \rvert\right)$, which immediately implies that \eqref{equation sharp radial moser} holds for $\alpha<\alpha_{N}$ if we replace $X_2$ by $Y_2$ in the exponential. Now, even though  $X_2$ and $Y_2$ behaves similarly near the origin, their behaviour changes drastically near the boundary of $\mathbb{B}_N$. This distinction is manifested in our next theorem. 
\begin{theorem} \label{weak moser}
     Let $N\geq 2$. Then 
     \begin{equation} \label{weak thor}
\sup_{\left \lbrace u \in W^{1,N}_{0,rad}\left(\mathbb{B}_{N}, w_{21}\right) :\left \lVert \nabla u \right\rVert_{N, w_{21}} \le 1\right\rbrace}  \int_{\mathbb{B}_N} e^{\alpha \lvert u \rvert^{N'} Y_2 X_1^{-1}} \, dx <\infty,
    \end{equation} iff $\alpha <\alpha_{N}$.  Moreover, for any $\alpha, \gamma > 0$ we have,
   
    \begin{equation}\label{moser failure}
\sup_{\left \lbrace u \in W^{1,N}_{0}\left(\mathbb{B}_{N}, w_{21}\right) :\left \lVert \nabla u \right\rVert_{N, w_{21}} \le 1\right\rbrace } \int_{\mathbb{B}_{N}} e^{\alpha \lvert u \rvert^{N'} Y_2^{\gamma}} \, dx = \infty.
    \end{equation}
\end{theorem}

We remark that, \eqref{moser failure} is a consequence of \eqref{equation non radial moser optimal} in Theorem \ref{non radial moser} with $f= Y_2X_{1}^{-1}$. However, in this special case we prove a stronger version of \eqref{moser failure}, which shows that one cannot replace $X_2$ by $Y_2$ in the optimal version of the Leray-Trudinger inequality \eqref{op-leray-trudinger}.
\begin{proposition} \label{lerray weak}
     For any $\alpha, \gamma > 0$ we have,
    \begin{equation}\label{lerray weak fail}
         \sup_{\left \lbrace u \in W^{1,N}_{0}\left(\mathbb{B}_{N}\right) : I_{N,\mathbb{B}_{N}}\left[u\right] \le 1\right\rbrace} \int_{B} e^{\alpha \lvert u \rvert^{N'}  Y_2^{\gamma}} \, dx = \infty.
    \end{equation}
\end{proposition}

Note that in the radial case we can find $\alpha>0$ such that \eqref{lerray weak fail} holds  with $\gamma = 1/(N-1)$. More precisely,  there exists a positive constant $\alpha \equiv \alpha(N)$ such that
    \begin{align}
        \sup_{\left \lbrace u \in C^1_{c,rad}\left(\mathbb{B}_N \right) : I_{N,\mathbb{B}_N}[u] \le 1 \right \rbrace} \int_{\mathbb{B}_N} e^{\alpha \lvert u \rvert^{N'}Y^{\frac{1}{N-1}}_2} \, dx < \infty.
    \end{align}
    This follows from the proof of \cite[Theorem 3.5]{di2024optimal}.
\par{\bf{Main results for $\beta \neq 1$}}:

In the spirit of the $\beta=1$ case, we now treat the $\beta \neq 1$ case. First we focus on $0<\beta <1$. Our first result here addresses the embedding of the entire space $W^{1,N}_{0}\left(\mathbb{B}_{N}, w_{1\beta}\right)$.
\begin{theorem} \label{moser weighted1}
   Let $0< \beta < 1$, $N\geq2$ and $\Omega$ be any bounded domain in $\mathbb{R}^N$ containing the origin. Then, $ W^{1,N}_{0}(\Omega, w_{1\beta}) \not \hookrightarrow L^{\psi_N}\left(\Omega\right)$, where $\psi_{N}(t)= \exp(t^{N'})-1$, for $t\geq0$. However, for all $\alpha \le N \omega_{N-1}^{\frac{1}{N-1}}(1-\beta)^{N'}$, we have
    \begin{align} \label{moser trudinger weight}
    \mathcal{J}_{N, \alpha, w_{1\beta}}:=  \sup_{\left \lbrace u \in W^{1,N}_{0}\left(\Omega, w_{1\beta}\right) :\left \lVert \nabla u \right\rVert_{N, w_{1\beta}} \le 1\right\rbrace}  \fint_{\Omega} e^{\alpha \lvert u \rvert^{N'} Y_1^{-\beta}} \, dx <\infty.
    \end{align}
Also the weight $Y_{1}^{-\beta}$, is optimal in the sense that if  $f: \Omega \to [0,\infty]$ is any measurable function, satisfying $f(x) Y_1^{\beta} \to \infty$ or $f(x) \to \infty$ as $x \to a$ for some $a \in \Omega$ or $a \in \partial \Omega$, respectively, then $ W^{1,N}_{0}(\Omega, w_{1\beta}) \not \hookrightarrow L^{\phi}\left(\Omega\right)$, where $\phi(x,t) = e^{t^{N'}f(x)}-1$, for $x\in \Omega$ and $t\geq 0$. In other words,
  \begin{equation} \label{equation non radial moser optimal beta}
 \sup_{\left \lbrace u \in W^{1,N}_{0}\left(\Omega, w_{1\beta}\right) :\left \lVert \nabla u \right\rVert_{N, w_{1\beta}} \le 1\right\rbrace}  \fint_{\Omega} e^{\alpha \lvert u \rvert^{N'} f} \, dx = \infty, \mbox{ for any } \alpha>0.
    \end{equation}
\end{theorem}
Note that, the conclusion $ W^{1,N}_{0}(\Omega, w_{1\beta}) \not \hookrightarrow L^{\psi_N}\left(\Omega\right)$ implies 
    \begin{align} \label{unboundedness of log}
        \sup_{\left \lbrace u \in W^{1,N}_{0}\left(\Omega, w_{1\beta}\right) :\left \lVert \nabla u \right\rVert_{N, w_{1\beta}} \le 1\right\rbrace} \int_{\Omega} e^{\alpha \lvert u \rvert ^{N'}} \, dx = \infty, 
    \end{align}
for any $\alpha>0$. This improves \eqref{ruf calanchi sup infinite} in the case $i=1$. Next we deal with the weight $w_{2\beta}$. Here, the embedding $W^{1,N}_{0}(\Omega, w_{2\beta}) \hookrightarrow L^{\psi_N}\left(\Omega\right)$ follows from \eqref{moser-trudinger}, since the both the weights $w_{2\beta}$ and $X_1^{-\beta}$ are greater than equal to 1. 
\begin{theorem} \label{moser weighted1/2}
    Let $0\le \beta < 1$, $N\geq2$ and $\Omega$ be any bounded domain in $\mathbb{R}^N$ containing the origin. Then, for all $\alpha \le N \omega_{N-1}^{\frac{1}{N-1}}(1-\beta)^{N'}$, we have
    \begin{align} \label{moser trudinger weight X}
    \mathcal{J}_{N, \alpha, w_{2\beta}}:=  \sup_{\left \lbrace u \in W^{1,N}_{0}\left(\Omega, w_{2\beta}\right) :\left \lVert \nabla u \right\rVert_{N, w_{2\beta}} \le 1\right\rbrace}  \fint_{\Omega} e^{\alpha \lvert u \rvert^{N'} X_1^{-\beta}} \, dx <\infty.
    \end{align} 
Also the weight $X_{1}^{-\beta}$, is optimal in the sense that if $f: \Omega \to [0,\infty]$ is any measurable function satisfying $f(x) X_1^{\beta} \to \infty$ as $x \to a$, for some $a \in \overline{\Omega}$, then $ W^{1,N}_{0}\left(\Omega, w_{2\beta}\right) \not \hookrightarrow L^{\phi}\left(\Omega\right)$, where $\phi(x,t) = e^{t^{N'}f(x)}-1$, for $x\in \Omega$ and $t\geq0$.
\end{theorem}
In radial case, we have the following sharp inequality in the spirit of \eqref{moser-trudinger}.

\begin{theorem} \label{moser weighted3}
    Let $N\geq 2$ and $0<\beta <1$. Then 
     \begin{equation} \label{equation sharp radial moser beta}
\mathcal{C}_{N, \alpha, w_{1\beta},rad}:=  \sup_{\left \lbrace u \in W^{1,N}_{0,rad}\left(\mathbb{B}_{N}, w_{1\beta}\right) :\left \lVert \nabla u \right\rVert_{N, w_{1\beta}} \le 1\right\rbrace}  \fint_{\mathbb{B}_N} e^{\alpha \lvert u \rvert^{N'}  X_1^{-\beta}} \, dx <\infty,
    \end{equation} iff $\alpha \leq \alpha_{N, \beta}\coloneqq N\omega_{N-1}^{\frac{1}{N-1}} (1-\beta) $. 
\end{theorem}
In fact, when $\alpha> \alpha_{N,\beta}$, we prove the following stronger version of \eqref{equation sharp radial moser beta} (see Proposition \ref{final battle} in Section \ref{proof moser idea}).
\begin{align*}
\sup_{\left \lbrace u \in W^{1,N}_{0,rad}\left(\mathbb{B}_{N}, w_{1\beta}\right) :\left \lVert \nabla u \right\rVert_{N, w_{1\beta}} \le 1\right\rbrace}  \fint_{\mathbb{B}_N} e^{\alpha \lvert u \rvert^{N'}  Y_1^{-\beta}} \, dx = \infty.
\end{align*}
Combining this with \eqref{equation sharp radial moser beta}, we obtain the sharp versions of \eqref{moser trudinger weight} and \eqref{moser trudinger weight X} in the space $W^{1,N}_{0,rad}\left(\mathbb{B}_{N}, w_{i\beta}\right)$, for $i=1,2$, as stated below. 
\begin{align}\label{equation sharp radial moser beta combined}
\sup_{\left \lbrace u \in W^{1,N}_{0,rad}\left(\mathbb{B}_{N}, w_{i\beta}\right) :\left \lVert \nabla u \right\rVert_{N, w_{i\beta}} \le 1\right\rbrace}  \fint_{\mathbb{B}_N} e^{\alpha \lvert u \rvert^{N'}  w_{i\beta}^{\frac{1}{N-1}}} \, dx <\infty,
\end{align}
iff $\alpha\leq \alpha_{N,\beta}$. As in the case of $\beta=1$, the failure of Poly\'a-Szeg\"o inequality with the weight $w_{i\beta}$ (see Lemma \ref{pz fails} in \Cref{app B}) obstructs us from establishing sharp versions of \eqref{moser trudinger weight} and \eqref{moser trudinger weight X}.  
\par We note that the inequalities \eqref{moser trudinger weight}, \eqref{moser trudinger weight X} and \eqref{equation sharp radial moser beta combined} imply the embedding of the spaces $ W^{1,N}_{0}\left(\Omega, w_{i\beta}\right)$ and $W^{1,N}_{0,rad}\left(\mathbb{B}_{N}, w_{i\beta}\right)$ into the Musielak-Orlicz spaces $L^{\phi_{w_{i\beta}}} (\Omega)$ and $L^{\phi_{w_{i\beta}}} (\mathbb{B}_N)$, respectively. These spaces are defined by the generalized Young's function 
$$\phi_{w_{i\beta}}(x,t) = \exp\left({t^{N'} w_{i\beta}^{\frac{1}{N-1}}}(x)\right)-1$$ for $t\geq 0$ and $x\in \Omega$ or $\mathbb{B}_N$. On the other hand, we obtain from Theorem \ref{ruf1} below, that $W^{1,N}_{0,rad}\left(\mathbb{B}_{N}, w_{i\beta}\right)$ is embedded into the Orlicz space $L^{\psi_{N, \beta}}(\mathbb{B}_N)$, where the Young's function is given by $\psi_{N, \beta}(t) = \exp(t^{\frac{N'}{1-\beta}})- 1$, for $t\geq 0$. 
\begin{thma} [Calanchi-Ruf \cite{RufClanchi2015, RufN}] \label{ruf1}
    Let  $0< \beta <1$ and $w_{i\beta}$ is defined as in \eqref{weights w} for $i=1,2$. Then we have the following result,
    \begin{align}\label{C-R beta sharp}
       \mathcal{D}_{N, \alpha, w_{i\beta},rad} \coloneqq \sup_{\{u \in W^{1,N}_{0,rad}(\mathbb{B}_{N},w_{i\beta}) : \|u\|_{N,w_{i\beta}}\le 1\}} \int_{\mathbb{B}_{N}} e^{\alpha \lvert u \rvert^{\frac{N'}{1-\beta}}}\, dx < \infty,
    \end{align}
    if and only if $\alpha \le \kappa_{N,\beta}\coloneqq N \left(\omega_{N-1}^{\frac{1}{N-1}}(1-\beta)\right)^{\frac{1}{1- \beta}}$.
\end{thma}
Now, in Proposition \ref{M-O embedding} of Section \ref{prel}, we establish that $L^{\psi_{N, \beta}}(\mathbb{B}_N)$ is embedded into $L^{\phi_{w_{i\beta}}} (\mathbb{B}_N)$. However, it is easy to see that, these two spaces are not equal in general. Nevertheless, the finiteness of $\mathcal{C}_{N, \alpha_{N,\beta}, w_{1\beta},rad}$ is equivalent to the finiteness of $\mathcal{D}_{N, \kappa_{N,\beta}, w_{1\beta},rad}$.
In fact,  the finiteness of $\mathcal{D}_{N, \kappa_{N,\beta}, w_{1\beta},rad}$ follows from that of $\mathcal{C}_{N, \alpha_{N,\beta}, w_{1\beta},rad}$ by the following estimate, which is a consequence of Lemma \ref{radial lemma} $(i)$ and $(ii)$.
\begin{align*}
\int_{\mathbb{B}_{N}} e^{\kappa_{N,\beta} \lvert u \rvert^{\frac{N'}{1-\beta}}}\, dx = \int_{\mathbb{B}_{N}} e^{\kappa_{N,\beta} \lvert u \rvert^{N'}\lvert u \rvert^{\frac{\beta N'}{1-\beta}}}\, dx \leq \int_{\mathbb{B}_{N}} e^{\alpha_{N,\beta} \lvert u \rvert^{N'}X_{1}^{-\beta}}\, dx. 
\end{align*}
However, we were unable to find a simpler proof of the reverse implication. 
\par Next, we consider the case where $\beta>1$. As noted earlier, we will consider only the weight $w_{2\beta}$, since $w_{1\beta} \notin A_{N}$. Calanchi and Ruf established in \cite{RufClanchi2015,RufN} that $W^{1,N}_{0,rad}\left(\mathbb{B}_N,w_{2\beta}\right) \hookrightarrow L^{\infty}\left(\mathbb{B}_N\right)$. However, the entire space $W^{1,N}_{0}\left(\mathbb{B}_N,w_{2\beta}\right)$ fails to embed into $L^\infty$, which is fairly straightforward to see. This motivates the search for an exponential integrability result. Indeed, we establish the following theorem.
\begin{theorem} \label{moser beta>1}
   Let $ \beta > 1$, $N\geq 2$, and $\Omega \subset \mathbb{R}^N$ be any bounded domain. Then there exists $\alpha \equiv \alpha(N,\beta)>0$ such that
    \begin{align} \label{moser trudger weight 2}
          \sup_{\left \lbrace u \in W^{1,N}_{0}\left(\Omega \setminus \lbrace 0\rbrace, w_{2\beta}\right) :\left \lVert \nabla u \right\rVert_{N, w_{2\beta}} \le 1\right\rbrace}  \fint_{\Omega} e^{\alpha \lvert u \rvert^{N'} X_1^{-\beta}} \, dx <\infty.
          \end{align}
    Moreover, if $\Omega$ contains the origin, then 
    \begin{align} \label{moser trudger weight 2 entire space}
          \sup_{\left \lbrace u \in W^{1,N}_{0}\left(\Omega, w_{2\beta}\right) :\left \lVert \nabla u \right\rVert_{N, w_{2\beta}} \le 1\right\rbrace}  \fint_{\Omega} e^{\alpha \lvert u \rvert^{N'} X_1^{-\beta}} \, dx =\infty.
          \end{align}
    for any  $\alpha > 0$.
\end{theorem}
Note that the distinction between \eqref{moser trudger weight 2} and \eqref{moser trudger weight 2 entire space} arises because the spaces $W^{1,N}_{0}\left(\Omega, w_{2\beta}\right) \neq W^{1,N}_{0}\left(\Omega \setminus \{0\}, w_{2\beta}\right)$ (see \Cref{not approximation} below for a proof).  

\par We now outline the major ideas utilized in this article. The primary focus is the proof of \eqref{equation non radial moser} in Theorem \ref{non radial moser}. To establish this result on balls, we first decompose the function u using spherical harmonics. Subsequently, we apply Lemma \ref{radial lemma}(i) to control the spherical mean of u, and we use the Poincaré inequality alongside Trudinger's technique \cite{Trudinger} to estimate the remaining components of u. This same approach was employed in \cite{di2024optimal} to establish the optimal Leray-Trudinger inequality \eqref{op-leray-trudinger}. In fact, the use of spherical decomposition to obtain a non-radial extension of a radial inequality is standard in the literature (see, for instance, \cite{VZ, Adi-Tin-Hardy-For-Dirac, Debdip-hardyrellich}).

\par The second main objective is establishing the finiteness of $\mathcal{C}_{N, \alpha_{N\beta}, w_{1\beta},rad}$ as in Theorem \ref{moser weighted3}. Despite sharing structural similarities, Theorem \ref{moser weighted3} and Proposition \eqref{moser} possess fundamental differences. Specifically, the bound on $\alpha_{N,\beta} |u|^{N'}X_1^{-\beta}$ obtained from Lemma \ref{radial lemma}(i) is not integrable, whereas the bound on $\alpha_N \lvert u \rvert^{N'} X_2 X_1^{-1}$ obtained from Lemma \ref{radial lemma}(iii) is integrable. To overcome this obstacle and prove the finiteness of $\mathcal{C}_{N, \alpha_{N\beta}, w_{1\beta},rad}$, we adapt Moser's classical method, originally developed to establish \eqref{moser-trudinger} with $\alpha=\alpha_N$. In outlining Moser's approach, consider an arbitrary radial function u within the unit ball of the Sobolev space $W_0^{1,N}(\mathbb{B}_N)$. Moser demonstrated the existence of a uniform $\delta_0>0$ with the property that, under an exponential change of coordinates, if u is contained within the ``$\delta_0$-neighbourhood'' (see equation $(10)$ in \cite{Moser}) of any member of a specific family of broken-line functions, the corresponding integrand in \eqref{moser-trudinger} remains uniformly bounded. This family of functions is now universally recognized as the \emph{Moser functions}. However, in the framework of Theorem \ref{moser weighted3}, employing exponential coordinates (or any alternative change of coordinates) is unsuitable. Consequently, we execute Moser's strategy without relying on any change of variables. Similar methods are used in the two recent articles \cite{Ruf-Bliss-MT, Ruf-Bliss-MT-dimN}, where they prove sharp exponential integrability similar to the one derived in Theorem \ref{moser weighted3} but with respect to an unweighted energy. 

\par Finally, when proving \eqref{equation non radial moser optimal} in Theorem \ref{non radial moser}, as well as analogous inequalities such as \eqref{equation non radial moser optimal beta} in Theorem \ref{moser weighted1}, we rely on the existence of two distinct Moser type functions. The first is supported near the origin and is determined by the weight, while the second is the original Moser function, supported away from the origin.

\par The article is organized as follows. Section \ref{prel} introduces necessary preliminaries, including Muckenhoupt $A_p$ weights, Musielak-Orlicz spaces, and fundamental inequalities used throughout this work. Section \ref{prop w-sob spaces} discusses the basic properties of the relevant weighted Sobolev spaces, featuring key density results established in Theorems \ref{approximation} and \ref{fail approximation}. In Section \ref{case beta 1}, we present our main results for the case $\beta=1$, proving Theorems \ref{non radial moser} and \ref{weak moser}, as well as Propositions \ref{moser} and \ref{lerray weak}. Section \ref{case gen beta} extends our analysis to the case $\beta\neq1$, where we prove Theorems \ref{moser weighted1}, \ref{moser weighted1/2}, and \ref{moser beta>1}. Section \ref{proof moser idea} is dedicated exclusively to the proof of Theorem \ref{moser weighted3}. Lastly, Appendices \ref{app A} and \ref{app B} provide a comparative analysis of the Leray and weighted energies and demonstrate the failure of the P\'olya-Szeg\"o inequality in weighted Sobolev spaces.

\section{Preliminaries}\label{prel} 
\subsection{ Muckenhoupt's \texorpdfstring{$A_p$}{}  weights} First we recall the definition of $A_p$ weight.
\begin{definition}
Let $w$ be a locally integrable non-negative function defined on $\mathbb{R}^N$ such that $0 < w < \infty$ a.e. on $\mathbb{R}^N$. Let $1 < p < \infty,$ we say that $w$ is an $A_p$ weight if there exists a positive constant $c_{p,w}$ such that,
\begin{align*}
    \left(\fint_{B} w(x) \, dx\right)\left(\fint_{B} w^{1/(1-p)}(x) \, dx\right)^{p-1} \le c_{p,w},
\end{align*}
for all balls $B \subset \mathbb{R}^N.$
\end{definition}

\begin{lemma}\label{ap weights lem}
    We have 
    \begin{itemize}
        \item[(i)]  $w_{2\beta} \in A_N$ for any $\beta \ge 0.$
        \item[(ii)] $w_{1\beta} \in A_N$ for any $0 \le \beta < 1.$
        \item[(iii)] $w_{1\beta} \notin A_N$ for any $ \beta \geq 1.$
    \end{itemize}
\end{lemma}
\begin{proof}
The proofs of $(i)$ and $(ii)$ follow from a similar argument to that used in \cite[Section 2, Proposition 1]{Sami}. We omit the details.  To prove $(iii)$, we note that 
        \begin{align*}
           \fint_{B(0,R_{\Omega})} w_{1\beta}^{-\frac{1}{N-1}} \, dx=  \fint_{B(0,R_{\Omega})}  \left(\ln \frac{ R_{\Omega}}{\lvert x \rvert}\right)^{-\beta}\, dx  &= N \int_{0}^{R_{\Omega}} \left(\ln \frac{R_{\Omega}}{r}\right)^{-\beta} r^{N-1} \, dr.
                    \end{align*}
        
      Thus 
      \begin{align*}
      \fint_{B(0,R_{\Omega})} w_{1\beta}^{-\frac{1}{N-1}} \, dx=N \int_{0}^{\infty} z^{-\beta} e^{-Nz} \, dz \ge c_N \int_{0}^{1} z^{-\beta} \, dz =\infty,
      \end{align*}
  for any $\beta \ge 1$ which implies $(iii)$. \\
\end{proof}

\subsection{Musielak-Orlicz Spaces}
The following definitions are taken from \cite[Section 2.3]{Var-Exp-Spaces-Diening-etal}.
\begin{definition}
A convex, left continuous function $\phi: [0,\infty) \to [0,\infty)$ with $\phi(0)=0,$ $\lim_{t\to 0+}\phi(t)=0$ and $\lim_{t\to \infty} \phi(t)= \infty$ is called a $\Phi$-function. In addition, it is called positive if $\phi(t)>0$ for all $t>0$.
\end{definition}

\begin{definition}
Let $\left(A, \Sigma, \mu \right)$ be a $\sigma-$finite, complete measure space. A real function $\phi: A\times [0,\infty)\to [0,\infty)$ is said to be a \emph{generalized} $\Phi$-function on $\left(A, \Sigma, \mu \right)$ if 
\begin{enumerate}
\item $\phi(y,\cdot)$ is a $\Phi$-function for every $y\in A$;
\item $y\mapsto \phi(y,t)$ is measurable for every $t\geq 0$. 
\end{enumerate}
\end{definition}

\begin{definition}\label{M-O space}
Let $\phi$ be a $\Phi$-function on the $\sigma-$finite, complete measure space $\left(A, \Sigma, \mu \right)$. Then the \emph{Musielak-Orlicz space} is denoted by $L^{\phi}\left(A,\mu\right)$ is the collection of all measurable functions $f:A\to \mathbb{R}$ for which 
\begin{align*}
\lim_{\lambda\to 0} \int_A \phi \left(y, \lambda \lvert f(y) \rvert \right) \, d\mu(y) =  0 .
\end{align*}

\end{definition}
The Musielak-Orlicz spaces $L^{\phi}\left(A,\mu\right)$ are Banach spaces when equipped with the norm 
\begin{align}\label{MO norm}
\left\lVert f\right\rVert_{\phi}: = \inf \left \lbrace \lambda>0:  \int_A \phi \left(y, \lambda^{-1} \lvert f(y) \rvert \right) \, d\mu(y) \leq 1 \right \rbrace.
\end{align}
If $\mu$ is the Lebesgue measure, we write $L^{\phi}(A)$ to denote $L^{\phi}(A, \, dx)$. We refer \cite[Theorem 2.3.13]{Var-Exp-Spaces-Diening-etal} for a proof.  We next recall \cite[Theorem 2.8.1]{Var-Exp-Spaces-Diening-etal}. 
\begin{theorem} \label{musileak embedding}
  Let $\left(A, \Sigma, \mu \right)$ be a $\sigma-$finite, complete measure space and $\phi$ and $\psi$ are generalized $\Phi$-function. Then $L^{\phi}\left(A, \mu\right) \hookrightarrow L^{\psi}\left(A, \mu\right)$ if and only if there exists $\lambda > 0$  and $f \in L^1\left(A,\mu \right)$ with $\left\lVert f \right\rVert_{1} \le 1$ such that,
    \begin{align}\label{m embed}
        \psi \left(x,\lambda^{-1} t\right) \le \phi\left(x,t\right) + f(x) \quad \text{ for a.e.} \quad x\in \Omega \quad \text{and all} \quad t \ge 0.
    \end{align}
\end{theorem}
This result is useful in proving the next Proposition.
\begin{proposition}\label{M-O embedding}
    Let $N \ge 2$, $0<\beta<1 $ and $\Omega$ be any domain in $\mathbb{R}^N$ containing the origin. Then we have the following embeddings for $i=1,2$:
    \begin{enumerate}[(i)]
        \item $L^{\psi_{N1}}(\Omega) \hookrightarrow L^{\phi_N}(\Omega)$,
        \item $L^{\psi_{N\beta}}(\Omega) \hookrightarrow L^{\phi_{w_{i\beta}}}(\Omega)$,
    \end{enumerate}
    where the functions are defined for $x \in \Omega$ and $t \ge 0$ as follows:
   \begin{align*} 
    \phi_N(x, t) &= \exp\left( t^{N'} X_2 X_1^{-1}(x)\right) - 1, & \phi_{w_{i\beta}}(x, t) &= \exp\left(t^{N'} w_{i\beta}^{\frac{1}{N-1}}(x)\right) - 1, \\ 
    \psi_{N1}(t) &= \exp\left(N e^{t^{N'}} - N \right) - 1, & \psi_{N\beta}(t) &= \exp\left(t^{\frac{N'}{1-\beta}}\right) - 1. 
\end{align*}
\end{proposition}
\begin{proof}
We will use \Cref{musileak embedding} to establish the results.  
    \par \emph{Proof of (i)}: Let $\lambda \geq 1$. By Young's inequality we have,
        $$ab \le e^a - a -1 + (1+b)\ln(1+b) -b, \mbox{ for any } a, b\ge 0.$$ For $x\in \Omega$ and $t\geq 0$, we estimate using this inequality
        \begin{align*}
            \phi_{N}(x, \lambda^{-1} t) &= e^{ t^{N'} \lambda^{-N'} X_2 X_1^{-1}} -1\\
            &\le \exp \left[ \lambda^{-N'} \left(e^{ t^{N'}} -1  + \left(1+ X_2 X_1^{-1}\right)\ln \left(1+ X_2 X_1^{-1}\right)\right)\right] -1\\
            &=\exp \left[\lambda^{-N'} \left(e^{ t^{N'}} -1\right)\right] \exp \left[\lambda^{-N'} \left(1+ X_2 X_1^{-1}\right)\ln \left(1+ X_2 X_1^{-1}\right)\right]-1\\
            &\le \frac{\exp \left[\lambda^{-N'} \left(Ne^{ t^{N'}} -N\right)\right]-1}{N} \\
            & \qquad \qquad + \frac{ \exp \left[\lambda^{-N'} N'\left(1+ X_2 X_1^{-1}\right)\ln \left(1+ X_2 X_1^{-1}\right)\right]-1}{N'}\\
            &\le \psi_{N1}(t) + f(x).
        \end{align*}
Now, we shall choose $\lambda$ large to conclude $\left\lVert f \right\rVert_1\leq 1$, which would establish \eqref{m embed} and thus conclude the proof of $(i)$. 
\par Since $X_2 X_1^{-1} \ge 1,$ so $\left(1+ X_2 X_1^{-1}\right)\ln \left(1+ X_2 X_1^{-1}\right) \le 2 X_2 X_1^{-1}\ln \left(2 X_2 X_1^{-1}\right)$. Additionally, $X_2 \ln \left(2 X_2 X_1^{-1}\right) \to 1$ as $x  \to 0$. Thus there exists $\delta > 0$ such that $2X_2 X_1^{-1}\ln \left(2 X_2 X_1^{-1}\right) \le 3X_1^{-1}$ in $ B(0,\delta).$ Therefore, we have
         \begin{align*}
            \int_{\Omega} f(x) \, dx &\leq  \int_{B\left(0,R_{\Omega}\right)} \left(\exp \left[\lambda^{-N'} N'\left(1+ X_2 X_1^{-1}\right)\ln \left(1+ X_2 X_1^{-1}\right)\right] -1 \right) \\
             &\le  \int_{B\left(0,R_{\Omega}\right)} \left(\exp \left(\lambda^{-N'} 2 N' X_2 X_1^{-1}\ln \left(2 X_2 X_1^{-1}\right) \right)-1\right)\\
             &\le \int_{B(0,\delta)} \left(e^{\lambda^{-N'} N' 3 X_1^{-1}} -1\right) + \int_{\delta < \lvert x \rvert < R_{\Omega}}  \left(e^ {\lambda^{-N'} 2 N' \frac{X_2}{X_1}\ln \left(\frac{2 X_2}{X_1} \right)}-1\right).
         \end{align*}
         Thus choosing $\lambda > 1$ large enough yields the bound $\left\lVert f \right\rVert_1\leq 1$. 
        \par \emph{Proof of (ii)}: We conside $0<\beta < 1$ and $i=2$. The proof for $i=1$ is similar. For $a,b \ge 0,$ we will use the Young's inequality $a^{1-\beta} b ^{\beta} \le (1-\beta) a + \beta b.$  Let $\lambda \ge 1$. Then for $x\in \Omega$ and $t\geq 0$ we have
        \begin{align*}
             \phi_{w_{i\beta}}(x, \lambda^{-1}t) &= e^{\lambda^{-N'} t^{N'} X_1^{-\beta}}-1\\
            &\le \exp \left[ \lambda^{-N'} \left((1-\beta)t^{\frac{N'}{1-\beta}} +  \beta X_1^{-1}\right) \right] - 1\\
            &\le (1-\beta)\left(e^{t^{\frac{N'}{1-\beta}}}-1\right) + \beta \left(e^{ \lambda^{-N'} X_1^{-1}}-1\right)\\
            &\le \psi_{N\beta}(t) + f(x),
        \end{align*}
        again we take $\lambda > 1$ large enough such that $\|f\|_{1} \le 1$, which establishes \eqref{m embed}. This completes the proof. \\
\end{proof}

\subsection{Some useful inequalities}
We need the following radial estimates. We refer Lemma $5$ of \cite{RufClanchi2015} and \cite{RufN} for proofs.
\begin{lemma}\label{radial lemma}
    Let $0< \beta <\infty$, $N\geq 2$ and $u \in C^1_{c,\mathrm{rad}}(\mathbb{B}_N)$. Then the following point wise estimates hold.
    \begin{enumerate}[(i)]
        \item For any $x \in \mathbb{B}_N$ and $0<\beta<1$, 
        \begin{align*}
        \lvert u(x) \rvert \le \frac{\left( \ln \frac{1}{\lvert x \rvert} \right)^{\frac{1-\beta}{N'}}}{\omega_{N-1}^{1/N}(1-\beta)^{1/N'}} \left( \int_{\mathbb{B}_N} \lvert \nabla u \rvert^N Y_1^{-\beta(N-1)} \, dx \right)^{1/N}.
        \end{align*}
        
        \item For any $x \in \mathbb{B}_N$ and $\beta \neq 1$
 \begin{align*}       
        \lvert u(x) \rvert \le \frac{\left\lvert\left(\ln \frac{e}{\lvert x \rvert}\right)^{1-\beta}-1\right\rvert^{1/N'}}{\omega_{N-1}^{1/N}\lvert 1-\beta \rvert^{1/N'}} \left(\int_{\mathbb{B}_N} \lvert \nabla u \rvert^N X_1^{-\beta(N-1)} \, dx \right)^{1/N}.
        \end{align*}
   
   \item For $\beta =1$ and  $x \in \mathbb{B}_N$ 
        \begin{align*}
        \lvert u(x) \rvert \le \frac{1}{\omega_{N-1}^{1/N}} \left(\ln \ln \frac{e}{\lvert x \rvert}\right)^{1/N'} \left(\int_{\mathbb{B}_N} \lvert \nabla u \rvert^N X_1^{-N+1} \, dx \right)^{1/N}.
  \end{align*} 
    \end{enumerate}
\end{lemma}

We also need the following Hardy inequality. See \cite[Lemma 2.1]{PsaradakisJFA} for a proof.
\begin{lemma} \label{psara}
     Let $N\geq 2$ and $\Omega$ be a bounded domain in $\mathbb{R}^N$ containing the origin. Then for all $\beta \neq 1$, $i=1, 2,$ and $u \in C^1_c(\Omega \setminus \{0\})$, we have the following inequality,
    \begin{align} \label{hardy type}
        \int_{\Omega} \frac{\lvert u \rvert ^N}{\lvert x \rvert^N} w_{i\beta}^{1-\frac{N^{\prime}} {\beta}} \, dx \le \left\lvert\frac{N'}{1- \beta}\right \rvert^{N} \int_{\Omega} \lvert \nabla u \rvert^N w_{i\beta} \, dx.
    \end{align}
\end{lemma}

\section{Properties of Weighted Sobolev spaces}\label{prop w-sob spaces}

In this section we prove some properties of the Weighted Sobolev spaces with the weights defined in \eqref{weights w}. Throughout this section, we assume $N\geq 2$ and $\Omega$ is an open bounded domain in $\mathbb{R}^N$. 

\begin{lemma}\label{contains lip functions}
    Let, $0<\beta <1$, if $i=1$ and $0<\beta <\infty$, if $i=2$. Then  $ W^{1,N}_{0}(\Omega, w_{i\beta})$ contains all locally Lipschitz functions vanishing at the boundary. 
\end{lemma}
\begin{proof}
For $\phi \in C_c^\infty(\Omega)$ we denote 
\begin{align*}
     \left \lVert \phi \right\rVert_{1,N,w_{i\beta},\Omega} \coloneqq  \left(\int_{\Omega} |\phi|^p w_{i\beta} \, dx \right)^{\frac{1}{p}} + \left(\int_{\Omega} \left \lvert \nabla \phi \right \rvert^p w_{i\beta} dx \right)^{\frac{1}{p}}  .
\end{align*}
Now, as in \cite{Kilpelainen-Weighted-Sob}, we define the function space $H^{1,N}(\Omega, w_{i\beta})$ as the completion of the set  $\lbrace \phi\in C^{\infty}(\Omega): \left \lVert \phi\right\rVert_{1,N,w_{i\beta},\Omega} <\infty \rbrace$ under the norm $\left \lVert \cdot \right\rVert_{1,N,w_{i\beta},\Omega}$. 
\par By Lemma \ref{ap weights lem}, $w_{i\beta} \in A_{N}$.  It then follows from \cite[Theorem 15.21]{Heinonen} that $w_{i\beta}$ is $N$-admissible (See \cite[Section 1.1]{Heinonen} for the definition) weights. Now let $u$ be any locally Lipschitz function which vanishes at the boundary. Since $w_{i\beta} \in L^1(\Omega)$ so,  $\left \lVert u \right\rVert_{1,N,w_{i\beta},\Omega} < \infty.$ Then by \cite[Theorem 2.5]{Kilpelainen-Weighted-Sob} it follows that, $u \in  H^{1,N}(\Omega, w_{i\beta}) .$ Now since $u$ vanishes on the boundary, so using \cite[Lemma 1.26]{Heinonen} we conclude the $u \in W^{1,N}_{0}\left(\Omega, w_{i\beta}\right)$. This completes the proof. \\
\end{proof}

Next we recall, \cite[Lemma 1.23]{Heinonen}. 

\begin{lemma}
  Let, $0<\beta <1$, if $i=1$ and $0<\beta <\infty$, if $i=2$.  Let $u,v \in W^{1,N}_{0}\left(\Omega, w_{i\beta}\right)$, then $\max \{u,v\}$ and $\min \{u,v \}$ are also in $W^{1,N}_{0}\left(\Omega, w_{i\beta}\right)$. Moreover, if $u \in W^{1,N}_{0}\left(\Omega, w_{i\beta}\right)$ non-negative then there exists a sequence of non-negative functions $u_m \in C^{\infty}_c\left(\Omega\right)$ such that $u_m \to u$ in $W^{1,N}_{0}\left(\Omega, w_{i\beta}\right)$ as $m\to \infty$.
\end{lemma}

\begin{corollary} \label{lattice}
  Let, $0<\beta <1$, if $i=1$ and $0<\beta <\infty$, if $i=2$. Then $u \in W^{1,N}_{0}\left(\Omega, w_{i\beta}\right)$ implies $\lvert u \rvert \in W^{1,N}_{0}\left(\Omega, w_{i\beta}\right).$
\end{corollary}
Now, we prove that compactly supported smooth functions vanishing near the origin are dense.
\begin{theorem} \label{approximation}
    Let, $0<\beta <1$, if $i=1$ and $0<\beta \leq 1$, if  $i=2$. Then 
    \begin{align*}
        W^{1,N}_{0}\left(\Omega, w_{i\beta}\right) = W^{1,N}_{0}\left(\Omega \setminus \lbrace 0 \rbrace, w_{i\beta} \right).
    \end{align*}
\end{theorem}
\begin{proof}
If $\Omega$ does not contain the origin then there is nothing to prove. Let $0 \in \Omega$, and $\delta>0$ be such that $B(0,\delta) \subset \Omega$. Since $C^{\infty}_{c}(\Omega\setminus\{0\}) \subset C^{\infty}_{c}(\Omega) $ so one direction follows trivially, we just need to check the other direction. 

  \par We first construct a sequence $u_k \in W^{1,N}_{0}\left(\Omega, w_{i\beta}\right)$ such that $0 \le u_k \le 1$, $u_k \equiv 1$ in a neighbourhood of the origin and $u_k \to 0$ in $W^{1,N}_{0}\left (\Omega, w_{i\beta}\right)$ as $k\to \infty$. The proof then follows from \cite[Theorem 2.43]{Heinonen}. 
  We define,
    \begin{align}\label{sub broken}
    u_k(x) \coloneqq \begin{cases}
       1,  \quad 0 \le \lvert x \rvert < \frac{1}{k}\\
    \frac{\ln \ln \frac{e \delta}{\lvert x \rvert}}{\ln \ln (k e \delta)}, \quad \frac{1}{k} \le \lvert x \rvert < \delta\\
    0, \quad \delta \le \lvert x \rvert .
    \end{cases}
\end{align}
Note that, by Lemma \ref{contains lip functions} $u_k \in W^{1,N}_{0}(\Omega, w_{i\beta})$ for each $k$.  Also, the weak  derivatives of $u_k$  are given by
\begin{align}\label{sub sub broken}
    \partial_{x_j}u_{k}(x) \coloneqq \begin{cases}
       0\quad \quad \quad \quad \quad \quad\quad \quad,  \quad 0 < \lvert x \rvert < \frac{1}{k}\\
    -\frac{1}{\left(\ln \frac{e \delta}{ \lvert x \rvert}\right)\left(\ln \ln (ke\delta)\right)} \frac{x_j}{\lvert x \rvert^2}, \quad \frac{1}{k} < \lvert x \rvert < \delta\\
    0\quad \quad \quad \quad \quad \quad\quad \quad, \quad \delta \le \lvert x \rvert .
    \end{cases}
\end{align}
for any $j=1, \dots, N$. 
Now, we will show that the integrals 
$$\int_{\Omega} \lvert u_k \rvert^N \left(\ln \frac{e R_{\Omega}}{\lvert x \rvert}\right)^{N-1} \mbox{ and } \int_{\Omega} \lvert \nabla u_k \rvert^N \left(\ln \frac{e R_{\Omega}}{\lvert x \rvert}\right)^{N-1} $$ 
converge to $0$ as $k \to \infty.$ This will complete our construction as 
$$\left(\ln \frac{R_{\Omega}}{\lvert x \rvert}\right)^{\beta(N-1)} \le \left(\ln \frac{e R_{\Omega}}{\lvert x \rvert}\right)^{\beta(N-1)} \le \left(\ln \frac{e R_{\Omega}}{\lvert x \rvert}\right)^{N-1}$$ 
for $0\le \beta <1$ and $x\in \Omega \setminus \lbrace 0 \rbrace$. As $u_k \to 0$ pointwise and $0 \le u_k \le 1$ so by dominated convergence theorem we have 
$$\lim_{k \to \infty}\int_{\Omega} \lvert u_k \rvert^N \left(\ln \frac{e R_{\Omega}}{\lvert x \rvert}\right)^{N-1} \, dx = 0.$$ 
Next, we consider
\begin{align*}
 \int_{\Omega} \lvert \nabla u_k \rvert^N \left(\ln \frac{e R_{\Omega}}{\lvert x \rvert}\right)^{N-1} \, dx &= \frac{\omega_{N-1}}{\left(\ln \ln (ke \delta)\right)^N}\int_{\frac{1}{k}}^{\delta} \frac{\left(\ln \frac{e R_{\Omega}}{r}\right)^{N-1}}{r\left(\ln \frac{e \delta}{r}\right)^{N}} \, dr\\
&= \frac{\omega_{N-1}}{\left(\ln \ln (k e \delta)\right)^N} \int_{1}^{\ln (k e \delta)} \frac{(v + \ln \frac{R_{\Omega}}{\delta})^{N-1}}{v^N} \, dv\\
&\leq  \frac{\omega_{N-1}2^{N}}{\left(\ln \ln (k e \delta)\right)^N} \int_{1}^{\ln (ke \delta)} \frac{(v^{N-1} + (\ln \frac{R_{\Omega}}{\delta})^{N-1})}{v^N} \, dv \\
&\to 0, \mbox{ as } k \to \infty. 
\end{align*}
This completes the proof. \\
\end{proof}

Finally, we prove that the same conclusion fails to hold for the weights $w_{2\beta}$, when $\beta>1$. 
\begin{theorem} \label{not approximation}
  For any $1<\beta <\infty$, we have 
     \begin{align} \label{fail approximation}
        W^{1,N}_{0}\left(\Omega, w_{2\beta}\right) \neq W^{1,N}_{0}\left(\Omega \setminus \{0\}, w_{2\beta}\right),
    \end{align}
    if $\Omega$ contains the origin.
\end{theorem}
\begin{proof}
    Since $\Omega$ contains the origin so there exists $\delta>0$ such that $B(0,\delta) \subset \Omega.$ If possible let $ W^{1,N}_{0}\left(\Omega, w_{2\beta}\right) = W^{1,N}_{0}\left(\Omega \setminus \{0\}, w_{2\beta} \right),$ holds. Then by combining  \Cref{psara} and Corollary \ref{lattice} we have,
    \begin{align} \label{fail hardy}
         \int_{\Omega} \frac{\lvert u \rvert ^N}{\lvert x \rvert^N} X_1^{N - (N-1) \beta} \, dx \le \left(\frac{N'}{\beta - 1}\right)^{N} \int_{\Omega} \lvert \nabla u \rvert^N \left(\ln \frac{e R_{\Omega}}{\lvert x \rvert}\right)^{\beta(N-1)},
    \end{align}
    for all $u \in  W^{1,N}_{0}\left(\Omega, w_{2\beta}\right) $. This gives a contradiction as $\frac{X_1^{N - (N-1) \beta} }{\lvert x \rvert^N} \notin L^1_{loc}(\Omega)$ for  any $\beta > 1$. This completes the proof.\\
    \end{proof}

\section{ Embeddings for \texorpdfstring{$\beta=1$}{}}\label{case beta 1}
In this section we prove \Cref{moser}, \Cref{non radial moser}, \Cref{weak moser}, and \Cref{lerray weak}.  We will use the following result frequently throughout the section, which follows from the standard decomposition of $L^2$ functions on $\mathbb{S}^{N-1}$ using spherical harmonics. See \cite[Chapter IV, Section 2]{stein-weiss} and \cite[Chapter 1]{Feng-Sph-Har} for details.
\begin{lemma}\label{decomposition lemma}
Let $u\in  W^{1,2}\left(\mathbb{B}_N\right)$, then employing the spherical coordinates $x= \left(r,\theta\right)$ in $\mathbb{B}_{N}$, we can decompose 
\begin{align}\label{spherical harmonic decomposotion}
u(x) = \sum_{k=0}^{\infty} u_k(r) f_k(\theta),
\end{align} 
 where $u_0(r) = \fint_{\mathbb{S}^{N-1}} u(r \theta) \, d \sigma\left(\theta\right)$, $f_0(\theta) = 1,$ $\left \lbrace f_k\right \rbrace_{k\in \mathbb{Z}^{+}}$ forms an orthonormal basis for $L^2(\mathbb{S}^{N-1})$ and satisfy 
 \begin{align*}
 -\Delta_{\mathbb{S}^{N-1}} f_k = k \left(k+N-2\right) f_k \mbox{ on } \mathbb{S}^{N-1}.
 \end{align*}

\end{lemma} 
 The proof of the following lemma follows from the derivation of \cite[Equation (28)]{di2024optimal}. We include the proof for convenience. 

 \begin{lemma}\label{estimates of angular part}
  There exists a constant $c_N>0$, such that for $u \in C^1_c(\mathbb{B}_N)$ we have,
    \begin{align}\label{estimates of angular part equation}
        \int_{\mathbb{B}_N} \left\lvert \nabla \left[(u-u_0)X_1^{-1+1/N}\right] \right \rvert^N \, dx \le c_N \int_{\mathbb{B}_N} \lvert \nabla u \rvert^N X_1^{-N+1} \, dx,
    \end{align}
    where $u_0$ is the spherical mean of $u$ i.e. $u_0(r) = \fint_{\mathbb{S}^{N-1}} u(r \theta) \, d \sigma\left(\theta\right)$.
\end{lemma}
\begin{proof}
    We write,
    \begin{align}
        \int_{\mathbb{B}_N} \lvert \nabla u \rvert^N X_1^{-N+1} \, dx &= \int_0^1 X_1^{-N+1}(r) r^{N-1} \int_{\mathbb{S}^{N-1}} \left[\left(\partial_r u \right)^2 + \frac{1}{r^2}\lvert \nabla_{\theta} u \rvert^2\right]^{\frac{N}{2}} d \sigma\left(\theta\right) dr \notag\\
        &\ge  \int_0^1 X_1^{-N+1}(r) r^{N-1} \int_{\mathbb{S}^{N-1}} \lvert\partial_r u \rvert^N \, d \sigma\left(\theta\right) \, dr\notag \\
        &+  \int_0^1 \frac{X_1^{-N+1}(r)}{r}  \int_{\mathbb{S}^{N-1}} \lvert \nabla_{\theta} u \rvert^N \, d \sigma \left(\theta\right) \, dr \coloneqq I_1 +I_2, \label{estimates of angular part int}
    \end{align}
    where we have used the fact that $\left(\lvert a \rvert + \lvert b \rvert\right)^p \ge \lvert a \rvert^p + \lvert b \rvert^p$ for any $p \ge 1$ and $a, b \in \mathbb{R}$.

Now we use the following elementary inequality  
    \begin{align*}
        \lvert b - a \rvert^N - \lvert a \rvert^N \ge \frac{1}{2^{N-1}-1} \lvert b \rvert^N - N \lvert a \rvert^{N-2} \langle a , b \rangle, \mbox{ for all } a,b \in \mathbb{R}^N
    \end{align*}
with $b = \partial_r u_0 $ and $a =- \partial_r(u-u_0)$ to derive:
\begin{align}
    \int_{\mathbb{S}^{N-1}} \lvert \partial_r u \rvert^N d \sigma\left(\theta\right) &\ge  \int_{\mathbb{S}^{N-1}} \lvert \partial_r u_0 \rvert^N d \sigma\left(\theta\right)\notag\\
    &\quad + \frac{1}{2^{N-1}-1}  \int_{\mathbb{S}^{N-1}} \lvert \partial_r (u-u_0) \rvert^N d \sigma\left(\theta\right)\notag \\
    &\quad \quad + N \int_{\mathbb{S}^{N-1}} \lvert \partial_r u_0 \rvert^{N-2}  \partial_r u_0  \partial_r (u-u_0) d \sigma\left(\theta\right). \label{estimates of angular part int1}
\end{align}

Now we use Lemma \ref{decomposition lemma} to write $u(x) = \sum_{k=0}^{\infty} u_k(r) f_k(\theta)$. Note that, the coefficients $u_k$ are given by $u_k(r) = \int_{\mathbb{S}^{N-1}} u(r \theta) f_k(\theta) d \sigma\left(\theta\right).$ 

Since $u \in C^1_c(\mathbb{B}_N)$ we have $u'_k(r) = \int_{\mathbb{S}^{N-1}} \partial_r u(r \theta) f_k(\theta) d \sigma\left(\theta\right)$ and we use Lemma \ref{decomposition lemma} with $\partial_r u$ to write
\begin{align*}
    \partial_r u(r \theta) &= \sum_{k=0}^{\infty} \left(\int_{\mathbb{S}^{N-1}} \partial_r (u(r \theta)) f_k(\theta) d \sigma\left(\theta\right)\right) f_k(\theta)\\
    &= \sum_{k=0}^{\infty} u_k'(r) f_k(\theta).
\end{align*}
Thus we have, $ \partial_r \left(u - u_0\right)  = \sum_{k=1}^{\infty} u_k'(r) f_k(\theta)$. Therefore we have,
\begin{align*}
    \int_{\mathbb{S}^{N-1}} &\lvert \partial_r u_0 \rvert^{N-2}  \partial_r u_0  \partial_r (u-u_0) d \sigma\left(\theta \right)\\
    &= \lvert u'_0(r) \rvert^{N-2} u_0'(r) \sum_{k=1}^{\infty} u_k'(r)  \int_{\mathbb{S}^{N-1}} f_k(\theta) \, d \sigma\left(\theta\right) = 0.
\end{align*}
This together with \eqref{estimates of angular part int1} implies,
\begin{align}
     \int_{\mathbb{S}^{N-1}} \lvert \partial_r u \rvert^N d \sigma\left(\theta\right) \ge \frac{1}{2^{N-1}-1}  \int_{\mathbb{S}^{N-1}} \lvert \partial_r (u-u_0) \rvert^N d \sigma (\theta).
\end{align}
Therefore from \eqref{estimates of angular part int} we have,
\begin{align*}
    &\int_{\mathbb{B}_N} \lvert \nabla u \rvert^N X_1^{-N+1} \, dx\\
    &\ge \frac{1}{2^{N-1}-1} \int_0^1 X_1^{-N+1}(r) r^{N-1}  \int_{\mathbb{S}^{N-1}} \left(\lvert \partial_r (u-u_0) \rvert^N +  \frac{1}{r^N}\lvert \nabla_{\theta} u \rvert^N \right) \, d \sigma\left(\theta\right) \, dr \\
    & \qquad \qquad  + \left(1 - \frac{1}{2^{N-1}-1} \right) I_2\\
    &\ge \frac{2^{1- N/2}}{2^{N-1}-1} \int_0^1 X_1^{-N+1}(r) r^{N-1}  \int_{\mathbb{S}^{N-1}} \left(\lvert \partial_r (u-u_0) \rvert^2 + \frac{1}{r^2}\lvert \nabla_{\theta} u \rvert^2 \right)^{N/2} \, d \sigma\left(\theta\right) \, dr \\
     & \qquad \qquad + \left(1 - \frac{1}{2^{N-1}-1} \right) I_2\\
     &= \frac{2^{1- N/2}}{2^{N-1}-1} \int_{\mathbb{B}_N} \lvert \nabla (u- u_0) \rvert^N X_1^{-N+1} \, dx + \left(1 - \frac{1}{2^{N-1}-1} \right) I_2,
\end{align*}
 where we have used the fact that $\left(\lvert a \rvert + \lvert b \rvert\right)^p \le 2^{p-1} \left(\lvert a \rvert^p + \lvert b \rvert^p\right)$ for any $p \ge 1$ and $a, b \in \mathbb{R}$. Now by Poincar\'e inequality on $\mathbb{S}^{N-1}$, which follows form \cite[Theorem 2.9]{hebey-book} combining with a contradiction-compactness argument, we have 
 \begin{align*}
     \int_{\mathbb{S}^{N-1}} \lvert \nabla_{\theta} u \rvert^N \,  d \sigma\left(\theta\right) \ge C_N  \int_{\mathbb{S}^{N-1}} \lvert u - u_0 \rvert^N d \sigma\left(\theta\right),
 \end{align*}
for some constant $C_N>0$.
Therefore we have,
\begin{align*}
     \int_{\mathbb{B}_N} &\lvert \nabla u \rvert^N X_1^{-N+1} \, dx\\
     &\ge c_N \left(\int_{\mathbb{B}_N} \lvert \nabla (u- u_0) \rvert^N X_1^{-N+1} \, dx + \int_{\mathbb{B}_N} \lvert x \rvert^{-N} \lvert u- u_0 \rvert^N X_1^{-N+1} \, dx \right)\\
     &\ge c_N \left(\int_{\mathbb{B}_N} \lvert \nabla (u- u_0) \rvert^N X_1^{-N+1} \, dx + \int_{\mathbb{B}_N} \lvert x \rvert^{-N} \lvert u- u_0 \rvert^N X_1 \, dx \right)\\
     &\ge c_N  \int_{\mathbb{B}_N} \left\lvert \nabla \left[(u-u_0)X_1^{-1+1/N}\right] \right \rvert^N \, dx,
\end{align*}
for some constant $c_N>0$. This completes the proof.\\
\end{proof}

\begin{proof}[Proof of \Cref{non radial moser}]
We first prove \eqref{equation non radial moser}. Let $u \in C^1_c(\Omega \setminus \{0\})$, extending by zero outside we may assume $u \in C^1_c\left(B\left(0,R_{\Omega}\right)\right) \setminus \{0\}$ and by scaling, we may assume that $R_{\Omega} = 1.$ We follow Trudinger's technique \cite{Trudinger}. Let $u_0$ be be the spherical mean as in \Cref{decomposition lemma}.    
  
   Let $q > N,$ then we have
    \begin{align}\label{non radial moser proof 1}
        &\left(\fint_{\mathbb{B}_N} \left \lvert u X_2^{1/N'} X_1^{-1/N'} \right\rvert^q \, dx\right)^{1/q} \notag \\
        &\qquad  \le \left(\fint_{\mathbb{B}_N} \left \lvert u_0 X_2^{1/N'} X_1^{-1/N'} \right\rvert^q \, dx\right)^{1/q} + \left(\fint_{\mathbb{B}_N} \left \lvert (u-u_0) X_2^{1/N'} X_1^{-1/N'} \right\rvert^q \, dx\right)^{1/q} \notag \\
        &\qquad  \coloneqq I_1 + I_2.    
         \end{align}
    We note $u-u_0 \in C_c^{1}\left(\mathbb{B}_N\right)$ and  estimate $I_2$ as follows,
    \begin{align}\label{non radial moser proof 2}
        I_2 &= \left(\fint_{\mathbb{B}_N} \left \lvert (u-u_0) X_2^{1/N'} X_1^{-1/N'} \right\rvert^q \, dx\right)^{1/q} \notag \\
        &\stackrel{\left(X_2 \leq 1\right)}{\le} \left(\fint_{\mathbb{B}_N} \left \lvert (u-u_0) X_1^{-1/N'} \right\rvert^q \, dx\right)^{1/q} \notag \\
        &\le c_1(N) q^{1- \frac{1}{N}} \left(\fint_{\mathbb{B}_N} \left \lvert \nabla \left[(u-u_0) X_1^{-1/N'} \right] \right\rvert^N \, dx\right)^{1/N}\notag \\
        &\stackrel{\eqref{estimates of angular part equation}}{\le} c_2(N) q^{1- \frac{1}{N}} \left(\fint_{\mathbb{B}_N} \left \lvert \nabla u \right\rvert^N X_1^{-N +1} \, dx\right)^{1/N},
    \end{align}
    where the second last inequality follows from the proof of  \cite[Theorem 1 ]{Trudinger}. Now we estimate $I_1$.
    Recall that  by \Cref{radial lemma} we have the estimate,
    \begin{equation}
        \lvert u_0(x) \rvert \le \frac{1}{\omega_{N-1}^{1/N}} \left(\ln \ln \frac{e}{\lvert x \rvert}\right)^{1/N'} \left(\int_{\mathbb{B}_N} \left \lvert \nabla u_0 \right\rvert^N X_1^{-N +1} \, dx\right)^{1/N}.
    \end{equation}
    This implies,
    \begin{align} \label{u_0 estimate}
         \lvert u_0(x) \rvert  X_2^{1/N'} \le \frac{1}{\omega_{N-1}^{1/N}} \left(\int_{\mathbb{B}_N} \left \lvert \nabla u_0 \right\rvert^N X_1^{-N +1} \, dx\right)^{1/N}.
    \end{align}
    We will show that,
    \begin{align}\label{grad u_0 estimate}
        \int_{\mathbb{B}_N} \left \lvert \nabla u_0 \right\rvert^N X_1^{-N +1} \, dx \le \int_{\mathbb{B}_N} \left \lvert \nabla u \right\rvert^N X_1^{-N +1} \, dx.
    \end{align}
    Note that,  $ u'_0 (r) = \fint_{\mathbb{S}^{N-1}} \partial_r u \, d \sigma\left(\theta \right).$ Thus we have,
    \begin{align*}
       \lvert \nabla u_0 \rvert^N = \left\lvert u'_0\right\rvert^N 
       & = \left \lvert \fint_{\mathbb{S}^{N-1}} \partial_r u \, d \sigma\left(\theta \right)\right \rvert^N\\
       &\le  \fint_{\mathbb{S}^{N-1}} \left \lvert \partial_r u \right\rvert^N \, d \sigma\left(\theta\right) \le \fint_{\mathbb{S}^{N-1}} \left(\left \lvert \partial_r u \right\rvert^2 +\frac{1}{r^2} \left\lvert \nabla_{\theta} u \right\rvert^2 \right)^{N/2} \, d \sigma\left(\theta\right).
    \end{align*}
    Thus we have,
    \begin{align*}
            \int_{\mathbb{B}_N} \left \lvert \nabla u_0 \right\rvert^N X_1^{-N +1} \, dx =    \omega_{N-1} \int_0^1 \lvert u_0'(r) \rvert^N r^{N-1} X_1^{-N+1} \, dr \le  \int_{\mathbb{B}_N} \left \lvert \nabla u \right\rvert^N X_1^{-N +1}.
    \end{align*}
    This proves \eqref{grad u_0 estimate}. Finally, we estimate $I_1$ as follows,
    \begin{align*}
        I_1 &= \left(\fint_{\mathbb{B}_N} \left \lvert u_0 X_2^{1/N'} X_1^{-1/N'} \right\rvert^q \, dx\right)^{1/q}\\
        &\stackrel{\eqref{u_0 estimate}}{\le} c_4(N) \left(\fint_{\mathbb{B}_N} X_1^{-q(1-1/N)} \, dx \right)^{1/q} \left(\int_{\mathbb{B}_N} \left \lvert \nabla u_0 \right\rvert^N X_1^{-N +1} \, dx\right)^{1/N}\\
        &\le c_4(N) \frac{e^{N/q}}{N^{1-1/N}} \left[\Gamma \left(1+q \frac{N-1}{N}\right)\right]^{1/q},
    \end{align*}
   To derive the last inequality we have used \eqref{grad u_0 estimate} along with $\left \lVert \nabla u \right\rVert_{N, w_{21}} \le 1$ and 
 $$\fint_{\mathbb{B}_N} X_1^{-q(1-1/N)} \, dx  \le \frac{e^{N}}{N^{q(1-1/N)}} \Gamma \left(1+q \frac{N-1}{N}\right). $$
Now combining this estimate of $I_1$ and the estimate of $I_2$ in \eqref{non radial moser proof 2} with \eqref{non radial moser proof 1} we derive for any $q>N$
\begin{align}\label{non radial moser proof 3}
\left(\fint_{\mathbb{B}_N} \left \lvert u X_2^{1/N'} X_1^{-1/N'} \right\rvert^q \, dx\right)^{\frac{1}{q}} \leq c_N \left( q^{1/N'} + e^{\frac{N}{q}} \Gamma^{\frac{1}{q}} \left(1+q \frac{N-1}{N}\right)\right).
\end{align}
Now for large $q$, we use Stirling's approximation $\Gamma(1+x) \sim x^xe^{-x}$ in \eqref{non radial moser proof 3} to obtain  
\begin{align*}
\left(\fint_{\mathbb{B}_N} \left \lvert u X_2^{1/N'} X_1^{-1/N'} \right\rvert^q \, dx\right)^{\frac{1}{q}} \leq c_N  q^{\frac{1}{N'}}.
\end{align*}
Now choosing $q=kN'$ and summing over $k\in \mathbb{N}$ we conclude the proof of \eqref{equation non radial moser}.
\par Next we derive \eqref{equation non radial moser optimal}.  \emph{First assume that  $a=0$ i.e. $f$ satisfies}    $fX_2^{-1}X_1 \to \infty$ as $ x \to 0.$ Let $\delta>0$ be such that $B(0,\delta) \subset \Omega.$ For $r_1>0$, consider the  function given by,
    \begin{align}
        u_{r_1}(x) = \begin{cases}
            \left(\ln \ln \frac{e \delta}{r_1}\right)^{1/N'}, \quad 0 \le \lvert x\rvert < r_1,\\
            \frac{\ln \ln \frac{e \delta}{\lvert x \rvert}}{\left(\ln \ln \frac{e\delta}{r_1}\right)^{1/N}}, \quad r_1 \le \lvert x \rvert <\delta,\\
            0 , \quad \lvert x \rvert \ge \delta.
        \end{cases}
    \end{align}

 Then by Lemma \ref{contains lip functions}, $u_{r_1} \in W^{1,N}_0 \left (\Omega, w_{21}\right)$ and 
  \begin{align*}
      \int_{\Omega} \left\lvert \nabla u_{r_1}(r) \right\rvert^N X_1^{-N+1} \, dx &= \frac{\omega_{N-1}}{\ln \ln \frac{e \delta}{r_1}} \int_{r_1}^{\delta} \frac{\left(\ln \frac{e R_{\Omega}}{r}\right)^{N-1}}{\left(\ln \frac{e \delta}{r}\right)^N} \frac{dr}{r}\\
      &=\frac{\omega_{N-1}}{\ln \ln \frac{e \delta}{r_1}} \int_{1}^{\ln \frac{e \delta}{r_1}} \frac{\left(\ln \frac{R_{\Omega}}{\delta} + z\right)^{N-1}}{z^N} \, dz\\
      &\le \frac{2^{N-2}\omega_{N-1}}{\ln \ln \frac{e \delta}{r_1}}  \int_{1}^{\ln \frac{e \delta}{r_1}} \left(\left(\ln \frac{R_{\Omega}}{\delta}\right)^{N-1} \frac{1}{z^N} + \frac{1}{z}\right) \le M.
  \end{align*}
Here $M \equiv M\left(N, \Omega\right)>0$ is a constant.  Then clearly, $v_{r_1} \coloneqq u_{r_1}/M^{1/N}$ belongs to $W^{1,N}_0 \left (\Omega, w_{21}\right) $ and satisfies $\left\lVert \nabla v_{r_1}\right\rVert_{N, w_{21}} \leq 1.$  

Now for any $\alpha>0$, denote $\tilde{\alpha}:= \alpha/M^{N'/N}$ and consider the integral
\begin{align*}
    I &\coloneqq \int_{\Omega} e^{\alpha \lvert v_{r_1} \rvert^{N'} f(x)} \, dx\ge\int_{B(0,r_1)} e^{\Tilde{\alpha} \ln \ln \frac{e \delta}{r_1} X_2 X_1^{-1} \frac{f(x)}{X_2 X_1^{-1}}} \, dx.
    \end{align*}
    Now since $f(x)X^{-1}_2X_1 \to \infty$ as $x \to 0$, so taking $r_1$ small enough yields $fX^{-1}_2X_1 \ge (N+2)/\Tilde{\alpha}$ in $B(0,r_1)$. Therefore, we have
    \begin{align*}
         I&\ge \int_{B(0,r_1)} e^{(N+2) \ln \ln \frac{e \delta}{r_1} X_2 X_1^{-1}} \, dx \\&= \omega_{N-1} \int_{0}^{r_1} e^{(N+2) \ln \ln \frac{e \delta}{r_1} \frac{\ln \frac{e R_{\Omega}}{r}}{1+ \ln \ln \frac{e R_{\Omega}}{r}}} r^{N-1} \, dr\\
    &= c_{N} \int_{\ln \frac{e R_{\Omega}}{r_1}}^{\infty} e^{(N+2) \ln \ln \frac{e \delta}{r_1} \frac{z}{1+ \ln z}} e^{-Nz} \, dz\\
    &\ge c_N \int_{\ln \frac{e R_{\Omega}}{r_1}}^{1+\ln \frac{e R_{\Omega}}{r_1}} \exp \left[z \left((N+2) \frac{\ln \ln  \frac{e R_{\Omega}}{r_1}}{1+ \ln z} - N\right) \right] \, dz\\
    &\ge  c_N \int_{\ln \frac{e R_{\Omega}}{r_1}}^{1+\ln \frac{e R_{\Omega}}{r_1}} \exp \left[z \left((N+2) \frac{\ln \ln  \frac{e R_{\Omega}}{r_1}}{1+ \ln \left(1+\ln \frac{e R_{\Omega}}{r_1}\right)} - N\right) \right] \, dz.
    \end{align*}
    Taking $r_1$ small enough yields $(N+2) \frac{\ln \ln  \frac{e R_{\Omega}}{r_1}}{1+ \ln \left(1+\ln \frac{e R_{\Omega}}{r_1}\right)} \ge N+1$. Hence we have,
    \begin{align*}
         I&\ge  c_N \int_{\ln \frac{e R_{\Omega}}{r_1}}^{1+\ln \frac{e R_{\Omega}}{r_1}} \exp \left[z \left(N+1- N\right) \right] \, dz= c_N (e-1) \frac{e R_{\Omega}}{r_1} \to \infty,
\end{align*}
as $r_1 \to 0$ proving \eqref{equation non radial moser optimal} in this case.

 \par \emph{Next we assume that $a\in \Omega \setminus \lbrace 0\rbrace$ and $f$ satisfies} $f(x)X_2^{-1}X_1 \to \infty$ as $x \to a$, which implies $f(x) \to \infty$ as $x\to a$. Let $0< \delta <\lvert a \rvert/2$ be such that $B(a,\delta) \subset \Omega \setminus \{0\}$. Now consider the family of functions $u_n$ defined as,
\begin{align*}
    u_n(x) = \begin{cases}
         \left(\ln n\right)^{1/N'} , \quad \lvert x-a \rvert \le \frac{\delta}{n} \\
         \frac{\ln \frac{\delta}{\lvert x -a \rvert}}{\left(\ln n \right)^{1/N}}, \quad \frac{\delta}{n} \le  \lvert x-a \rvert < \delta\\
        0, \quad \lvert x - a \rvert \ge \delta
    \end{cases}
\end{align*}
Again by Lemma \ref{contains lip functions}, $u_{n} \in W^{1,N}_0 \left (\Omega, w_{21}\right)$. Also, using the fact $|x|>|a|/2$, we derive $\left\lVert \nabla u_{n}\right\rVert_{N, w_{21}} \leq M$, for some constant $M\equiv M \left(N, \Omega, a\right)>0$.  So we normalize $u_n$ by considering $v_{n}=u_{n}/M$ to have $\left\lVert \nabla v_{n}\right\rVert_{N, w_{21}} \leq 1$. Now for any $\alpha>0$, we denote $\tilde{\alpha}:= \alpha/M^{N'}$ and estimate, 
 \begin{align*}
      \int_{\Omega} e^{\alpha \lvert v_n \rvert^{N'} f(x)} \, dx \ge \int_{B\left(a,\frac{\delta}{n}\right)} e^{\tilde{\alpha} f(x) \ln n } \, dx.
     \end{align*}
     Since $f(x) \to \infty$ as $x \to a,$ so taking $n$ large enough yields $f(x) \ge \frac{N+1}{\tilde{\alpha}}$ on $B\left(a,\frac{\delta}{n}\right).$ Thus we have,
     \begin{align*}
        \int_{\Omega} e^{\alpha \lvert v_n \rvert^{N'} f(x)} \, dx \ge \int_{B\left(a,\frac{\delta}{n}\right)} e^{(N+1) \ln n} \, dx\to \infty \mbox{ as } n \to \infty.
 \end{align*}
This proves \eqref{equation non radial moser optimal} in this case.
 \par \emph{Finally, we assume $a\in \partial \Omega$ and  $f$ satisfies} $f(x)X_2^{-1}X_1 \to \infty$ as $x \to a,$ which implies $f(x) \to \infty$ as $x\to a$. Let $x_p \in \Omega$ be such that $x_p \to a$ as $p \to \infty$.  Also, suppose $0<\delta_p < \min\{1/4, \lvert x_p \rvert/2\}$ is small enough such that $B(x_p, \delta_p) \subset \Omega$. Now consider the family of functions given by,
\begin{align*}
       v_{n,p}(x) \coloneqq \begin{cases}
        \left(\ln n\right)^{1/N'}, \quad \lvert x - x_p \rvert < \frac{\delta_p}{n}\\
        \frac{\ln  \frac{\delta_p}{\lvert x - x_p \rvert}}{\left(\ln n\right)^{1/N}}, \quad \frac{\delta_p}{n} \le \lvert x - x_p \rvert < \delta_p\\
        0 \quad \quad \quad, \quad \delta_p \le \lvert x - x_p \rvert.
    \end{cases}
    \end{align*}
 By Lemma \ref{contains lip functions}, $v_{n,p} \in W^{1,N}_0 \left (\Omega, w_{21}\right)$. Since $x_p \to a$ as $p \to \infty$, so there exists $k_1 \in \mathbb{N}$ such that $\lvert x_p \rvert \ge \lvert a \rvert /2>0,$ for all $p \ge k_1$. Thus for $x \in B(x_p, \delta_p)$ and $p\ge k_1$ we have
\begin{align*}
\lvert x \rvert \ge \lvert x_p \rvert - \lvert x-x_p \rvert &\ge \lvert x_p \rvert - \delta_p\ge \lvert x_p \rvert - \frac{\lvert x_p \rvert}{2} = \frac{\lvert x_p \rvert}{2} \ge  \frac{\lvert a \rvert}{4}.
\end{align*}
This implies $\left\lVert \nabla v_{n, p}\right\rVert_{N, w_{21}} \leq M$, for all $n\geq 1$ and $p\geq k_1$, where  $M\equiv M \left(N, \Omega, a\right)>0$ is a constant. Now we consider $u_{n,p}(x) =  v_{n,p}/M$. Then $\left\lVert \nabla u_{n, p}\right\rVert_{N, w_{21}} \leq 1$. Now for any $\alpha>0$, we denote $\tilde{\alpha}:= \alpha/M^{N'}$ and consider the integral,
\begin{align*}
    I & \coloneqq  \int_{\Omega} e^{\alpha \lvert v_{n,p} \rvert^{N'} f(x)} \, dx \geq  \int_{B\left(x_p,\delta_p/n\right)} e^{\tilde{\alpha} \ln n f(x)} \, dx.
    \end{align*}
 Since $f(x) \to \infty$ as $x \to a$ so $f \ge (N+1)/\tilde{\alpha}$ in $B(a, \varepsilon) \cap \Omega$ for $\varepsilon>0$ small enough. Now we choose $k_2\geq k_1$ such that  for any $n, p \ge k_2$ and $x\in B\left(x_p, \delta_p/n\right)$ we have $\lvert x - a \rvert \le \lvert x - x_p \rvert + \lvert x_p -a \rvert \le \delta_p/n + \varepsilon/4 \le 1/4n + \varepsilon/4 \le  \varepsilon/2$.
\par Thus in particular we have,  $f \ge (N+1)/\tilde{\alpha}$ in $B\left(x_p, \frac{\delta_p}{n}\right)$, for any $n, p \geq k_2$, which  implies
\begin{align*}
    I&\ge \int_{B\left(x_p,\delta_p/n\right)} e^{(N+1) \ln n} \, dx \to \infty \mbox{ as } n \to \infty.
\end{align*}
This completes the proof of \eqref{equation non radial moser optimal}. \\
\end{proof}

\begin{proof}[Proof of Proposition \ref{moser}]
First we consider, $\alpha \le  N \omega_{N-1}^{\frac{1}{N-1}}$. We only need to consider radial functions $u\in C^1_{c}\left(\mathbb{B}_N \setminus \lbrace 0\rbrace \right).$ Using $item (iii)$ of Lemma \ref{radial lemma}, we have  
     \begin{align*}
          \int_{\mathbb{B}_N} e^{\alpha \lvert u \rvert^{\frac{N}{N-1}} X_2 X_1^{-1}} \, dx &= \omega_{N-1} \int_{0}^{1} e^{\alpha u^{N'}(r) X_2(r) X_1^{-1}(r)} r^{N-1} \, dr \\
    &  \le \omega_{N-1} \int_{0}^{1} e^{\alpha\omega_{N-1}^{-\frac{1}{N-1}}\frac{\ln \ln\frac{e}{r}}{1 + \ln \ln\frac{e}{r}}\ln\frac{e}{r}} r^{N-1} \, dr\\
  &= \omega_{N-1} e^N \int_{1}^{\infty} e^{\alpha \omega_{N-1}^{-\frac{1}{N-1}}\frac{\ln z}{1+ \ln z}z - Nz} \, dz.
     \end{align*}
     The above integral converges if and only if $\alpha \omega_{N-1}^{-\frac{1}{N-1}} \le N.$
    \par Next we consider, $\alpha >  N \omega_{N-1}^{\frac{1}{N-1}}$. For $0<r_1<1$, we define
    \begin{align}\label{radial ltm function}
       \eta_{r_1}(x) \coloneqq \begin{cases}
        \frac{1}{\omega_{N-1}^{1/N}} \left(\ln \ln \frac{e}{r_1}\right)^{1/N'}, \quad 0 \le \lvert x \rvert < r_1\\
        \frac{1}{\omega_{N-1}^{1/N}} \frac{\ln \ln \frac{e}{\lvert x \rvert}}{\left(\ln \ln \frac{e}{r_1}\right)^{1/N}}, \quad r_1 \le \lvert x \rvert \le 1.
    \end{cases}
    \end{align} 
By Lemma \ref{contains lip functions}, $\eta_{r_1} \in W^{1,N}_0\left(\mathbb{B}_N, w_{21}\right)$ and
\begin{align*}
    \nabla \eta_{r_1}(x) = \begin{cases}
        0, \quad 0 < \lvert x\rvert < r_1\\
        -\frac{1}{\omega_{N-1}^{1/N}} \frac{1}{\left(\ln \ln \frac{e}{r_1}\right)^{1/N}} \frac{x}{\lvert x\rvert^{2} \ln \frac{e}{\lvert x \rvert}}, \quad r_1 < \lvert x \rvert < 1.
        \end{cases}
\end{align*}

This implies,
\begin{align*}
 \int_{\mathbb{B}_N} \left\lvert \nabla \eta_{r_1} \right\rvert^N X_1^{-N+1} \, dx &= \frac{1}{\ln \ln \frac{e}{r_1}} \int_{r_1}^{1} \frac{r^{N-1}}{r^{N} \left(\ln \frac{e}{r}\right)^N} \left( \ln \frac{e}{r}\right)^{N-1} \, dr\\
    &= \frac{1}{\ln \ln \frac{e}{r_1}} \int_{r_1}^{1} \frac{1}{r \ln \frac{e}{r}} \, dr =1.
\end{align*}
On the other hand, the integral,
\begin{align}\label{moser prop 1}
\frac{1}{\omega_{N-1}}\int_{\mathbb{B}_N} e^{\alpha \left\lvert \eta_{r_1} \right\rvert^{N'} X_2X_1^{-1}}\geq   \int_{0}^{r_1} e^{\alpha \eta_{r_1}^{N'}(r) X_2(r) X_1^{-1}(r)} r^{N-1} \, dr \eqqcolon I_{r_1}.
\end{align}

Note that $X_2 X_1^{-1}$ is a decreasing function in $(0,1)$. So we have
$$X_2(r) X_1^{-1}(r) \ge \frac{\ln \frac{e}{r_1}}{1 + \ln \ln \frac{e}{r_1}} \eqqcolon f(r_1), \text{ for all } r \in (0,r_1).$$
Thus,
\begin{align*}
    I_{r_1}  \ge \int_{0}^{r_1} e^{ \alpha \omega_{N-1} ^ { - \frac{1}{N-1}} \ln \frac{e}{r_1}f(r_1)} r^{N-1} \, dr.
\end{align*}
Since $\alpha > N \omega_{N-1}^{\frac{1}{N-1}} $, so for some $\varepsilon\equiv \varepsilon(\alpha, N)>0 $,  $\alpha \omega_{N-1} ^ { - \frac{1}{N-1}}= N+\varepsilon$. Thus
\begin{align*}
    I_{r_1} & \ge \frac{1}{N} e^{(N+\varepsilon) \ln \frac{e}{r_1} f(r_1)} r_1^N = \frac{e^N}{N}e^{\left((N+ \varepsilon) f(r_1)-N\right) \ln \frac{e}{r_1}}  \to \infty 
\end{align*}
as $r_1 \to 0.$ So the proof of \eqref{equation sharp radial moser} follows form \eqref{moser prop 1}.
 \par The existence of maximizing sequence follows from the dominated convergence theorem. We skip the details. This completes the proof. \\
\end{proof}

\begin{proof}[Proof of \Cref{weak moser}]
    We will only prove that the L.H.S. of \eqref{weak thor} is infinite for $\alpha \ge  N \omega_{N-1}^{\frac{1}{N-1}}.$ The rest follows from $item (iii)$ of Lemma \ref{radial lemma}. To this end, we again consider the same family of functions as defined in \eqref{radial ltm function}. Then we have
\begin{align}\label{weak moser thm 1}
\frac{1}{\omega_{N-1}}\int_{\mathbb{B}_N} e^{\alpha \left\lvert \eta_{r_1} \right\rvert^{N'} Y_2X_1^{-1}}
&\geq  \int_{0}^{r_1} e^{ \alpha \omega_{N-1}^{-\frac{1}{N-1}}  \frac{\ln \ln \frac{e}{r_1}}{\ln \ln \frac{e}{r}}X_1^{-1}(r)} r^{N-1} \, dr \notag \\ 
&\geq  \int_{0}^{r_1} e^{N \frac{\ln \ln \frac{e}{r_1}}{\ln \ln \frac{e}{r}} \ln \frac{e}{r}} r^{N-1} \, dr \eqqcolon I_{r_1}.
\end{align}
Now using the change of variable $z = \ln \frac{e}{r}$ and denoting $\ln \frac{e}{r_1}$ as $z_1$, we obtain
\begin{align*}
    I_{r_1} = e^N \int_{z_1}^{\infty} e^{N \frac{\ln z_1}{\ln z}z} e^{-Nz} \, dz.
\end{align*}

Again using the substitution $v= \ln z$ and denoting $\ln z_1$ as $v_1$, we ended up with,
\begin{align*}
    I_{r_1} &= e^N \int_{v_1}^{\infty} \exp \left(\frac{N v_1}{v}e^v - N e^v + v\right) \, dv \ge e^N \int_{v_1}^{v_1 +1} \exp \left(N e^v\frac{ v_1 - v}{v} + v\right) \, dv.
\end{align*}
Now, putting $w = v -v_1 $, we have.
\begin{align*}
    I_{r_1}& \ge \int_{0}^{1} \exp \left(-N w \frac{e^{v_1 + w}}{v_1 + w} +v_1 +w \right) \, dw\\
    & \ge \int_{0}^{1} \exp \left(-N w \frac{e^{v_1 + w}}{v_1} +v_1 +w \right) \, dw\\
    &=  \int_{0}^{1} \exp \left(-N w \frac{e^{v_1 + w}}{v_1} +v_1 +w \right) \frac{v_1}{N(w+1)} \frac{N(w+1)}{v_1}\, dw\\
    & \ge \frac{v_1}{2 N}  \int_{0}^{1} \exp \left(-N w \frac{e^{v_1 + w}}{v_1} +v_1 +w \right) \frac{N(w+1)}{v_1}\, dw.
\end{align*}
Finally, using the change of variable $Y = N e^{v_1 + w}w/v_1 $, we have
\begin{align*}
    I_{r_1} &\ge \frac{v_1}{2 N}  \int_{0}^{ N\frac{e^{v_1 +1}}{v_1}} e^{-Y} \, dY= \frac{v_1}{2 N} \left[1 - \exp \left(- N\frac{e^{v_1+1}}{v_1}\right) \right],
\end{align*}
which goes to $\infty$ as $v_1 \to \infty.$ Now taking $r_1 \to 0$ in \eqref{weak moser thm 1} and noticing that $\lim_{r_1\to 0} v_1 =\infty$, we conclude the proof.\\ 
\end{proof}

\begin{proof}[Proof of \Cref{lerray weak}]
  We consider the family of functions given by,
\begin{align*}
       v_{n,p}(x) \coloneqq \begin{cases}
        \frac{1}{\omega_{N-1}^{1/N}} \left(\ln n\right)^{1/N'}, \quad \lvert x - x_p \rvert < \frac{p}{n}\\
        \frac{1}{\omega_{N-1}^{1/N}} \frac{\ln  \frac{p}{\lvert x - x_p \rvert}}{\left(\ln n\right)^{1/N}}, \quad \frac{p}{n} \le \lvert x - x_p \rvert < p \\
        0 \quad \quad \quad, \quad p \le \lvert x - x_p \rvert
    \end{cases}
    \end{align*}
where $x_p=(1-p,0,0,...0) \in \mathbb{R}^N, n \ge 3$ and $p \in (0, 1)$
will be chosen later depending upon $\alpha, \beta$. Clearly, $v_{n,p} \in W^{1,N}_0\left(\mathbb{B}_N\right)$ and  
\begin{align*}
      \nabla v_{n,p}(x) \coloneqq \begin{cases}
       0, \quad \lvert x - x_p \rvert < \frac{p}{n}\\
      - \frac{1}{\omega_{N-1}^{1/N}\left(\ln n\right)^{1/N}}  \frac{x-x_p}{\lvert x - x_p \rvert^2}, \quad \frac{p}{n} < \lvert x - x_p \rvert < p\\
        0 \quad \quad \quad, \quad p < \lvert x - x_p \rvert.
    \end{cases}
    \end{align*}
    Thus we have
    \begin{align*}
        \int_{\mathbb{B}_N} \left \lvert \nabla v_{n,p}(x) \right \rvert^N \, dx &= \int_{\frac{p}{n} < \lvert x - x_p \rvert < p} \frac{1}{ \omega_{N-1} \ln n} \frac{1}{\lvert x-x_p \rvert^N} \, dx =1.
    \end{align*}
Hence from the definition \eqref{reduced energy leray}, we have 
    \begin{align*}
       I_{N, \mathbb{B}_{N}}[v_{n,p}] \le \int_{\mathbb{B}_N} \left \lvert \nabla v_{n,p}(x) \right \rvert^N \, dx & =1.
    \end{align*}
    Now note that, for $x \in B_{p/n}(x_p)$, we have  $\lvert x \rvert \ge \lvert x_p \rvert - \lvert x-x_p \rvert > 1-p - \frac{p}{n}> 1-\frac{4}{3} p$, as $n\geq 3$. It follows that,
\begin{align} \label{third}
    \frac{1}{\ln \ln \frac{e}{\lvert x \rvert}} > \frac{1}{\ln \ln \frac{e}{1-\frac{4}{3}p}}.
\end{align}

So, for any $\alpha,\gamma>0$, using the notation $\tilde{\alpha} = \alpha \omega^{\frac{-1}{N-1}}$, we estimate
\begin{align}\label{final est prop weak leray trudinger}
    \int_{\mathbb{B}_N} e^{\alpha \lvert v_{n,p} \rvert^{\frac{N}{N-1}} \frac{1}{\left(\ln \ln \frac{e}{\lvert x \rvert}\right)^{\gamma}} } \, dx 
    &\ge \int_{B_{p/n}(x_p)} \exp \left(   \frac{\tilde{\alpha}\ln n}{\left(\ln \ln \frac{e}{\lvert x \rvert}\right)^{\gamma}}\right) \, dx \notag \\
    & \stackrel{\eqref{third}}{\ge} \int_{B_{p/n}(x_p)} \exp \left( \frac{\tilde{\alpha}\ln n}{\left(\ln \ln \frac{e}{1-\frac{4}{3}p}\right)^{\gamma}}\right) \, dx\notag \\
    &=  N \omega_{N-1}p^N  \exp \left( \frac{\tilde{\alpha}\ln n}{\left(\ln \ln \frac{e}{1-\frac{4}{3}p}\right)^{\gamma}} - N\ln n\right).
\end{align}
Now we choose $p$ such that
\begin{align*}
    p < \frac{3}{4} \left[1 - \exp \left(1- e^{\left(\frac{\Tilde{\alpha}}{N+1}\right)^{\frac{1}{\gamma}}}\right)\right] .
\end{align*}
This implies 
$\Tilde{\alpha}/  \left(\ln \left(1- \ln \left(1-\frac{4}{3}p\right)\right)\right)^{\gamma} > N+1.$ 
Finally, with this choice of $p$, letting $n \to \infty$  in \eqref{final est prop weak leray trudinger} we conclude the proof of \eqref{lerray weak fail}. \\
\end{proof}

\section{Embeddings for \texorpdfstring{$\beta\neq 1$}{}}\label{case gen beta}
In this section, we will prove \Cref{moser weighted1}, \Cref{moser weighted1/2}, and \Cref{moser beta>1}.
\begin{proof} [Proof of \Cref{moser weighted1}]
First we prove that $ W^{1,N}_{0}(\Omega, w_{1\beta}) \not \hookrightarrow L^{\psi_N}\left(\Omega\right)$.

For this, fix $\alpha > 0.$ Recall that,  $R_{\Omega} \coloneqq \sup_{x \in \Omega} \lvert x \rvert = \sup_{x \in \partial \Omega} \lvert x \rvert$. Thus there exists $a \in \partial \Omega$ such that $\lvert a \rvert = R_{\Omega}.$ So, $\ln \left(R_{\Omega}/\lvert x \rvert\right) \to 0$ as $x \to a$. We choose a sequence  $x_p \in \Omega$ such that $x_p \to a$ as $p \to \infty$. Let  $\delta_p>0$ be such that $\delta_p \to 0$ as $p \to \infty$ , and  $B(x_p, \delta_p) \subset \Omega$. Now consider the family of functions given by,
\begin{align*}
       v_{n,p}(x) \coloneqq \begin{cases}
       \left(\frac{N+1}{\alpha}\right)^{\frac{1}{N'}}   \left(\ln n\right)^{\frac{1}{N'}}, \quad \lvert x - x_p \rvert < \frac{\delta_p}{n}\\
       \left(\frac{N+1}{\alpha}\right)^{\frac{1}{N'}}   \frac{\ln  \frac{\delta_p}{\lvert x - x_p \rvert}}{\left(\ln n\right)^{\frac{1}{N}}}, \quad \frac{\delta_p}{n} \le \lvert x - x_p \rvert < \delta_p\\
        0 \quad \quad \quad, \quad \delta_p \le \lvert x - x_p \rvert.
    \end{cases}
    \end{align*}
Then by Lemma \ref{contains lip functions}, $v_{n,p}\in  W^{1,N}_{0}\left(\Omega, w_{1\beta}\right)$. Also,
\begin{align*}
      \nabla v_{n,p}(x) \coloneqq \begin{cases}
       0, \quad \lvert x - x_p \rvert < \frac{\delta_p}{n}\\
       - \left(\frac{N+1}{\alpha}\right)^{\frac{1}{N'}} \frac{1}{\left(\ln n\right)^{1/N}}  \frac{x-x_p}{\lvert x - x_p \rvert^2}, \quad \frac{\delta_p}{n} < \lvert x - x_p \rvert < \delta_p\\
        0 \quad \quad \quad, \quad \delta_p < \lvert x - x_p \rvert.
    \end{cases}
    \end{align*}
Since $\ln \left(R_{\Omega}/\lvert x \rvert \right)\to 0$ as $x \to a$ so we have, $w_{1\beta} \le 1/\omega_{N-1}\left(\alpha/(N+1)\right)^{1/N'}$ in $B(a,\varepsilon) \cap \Omega$, for some $\varepsilon > 0$ small enough.

Since $x_p \to a$ and $\delta_p \to 0$, as $p \to \infty,$ there exists $k_1 \in \mathbb{N}$ such that for all $p \ge k_1$ we have $\lvert x_p - a \rvert < \varepsilon/2$ and $\delta_p < \varepsilon/2.$ Thus for $x \in B(x_p,\delta_p),$ we have $\lvert x - a \rvert \le \lvert x - x_p \rvert  + \lvert x_p - a \rvert \le \delta_p +  \lvert x_p - a \rvert < \varepsilon. $ Therefore, in $B(x_p,\delta_p)$ we have, $w_{1\beta} \le 1/\omega_{N-1}\left(\alpha/(N+1)\right)^{1/N'}$. 
     Thus we have,
    \begin{align*}
        \int_{\Omega} \left \lvert \nabla v_{n,p}(x) \right \rvert^N w_{1\beta} \, dx&= \int_{\frac{\delta_p}{n} < \lvert x - x_p \rvert < \delta_p}  \left(\frac{N+1}{\alpha}\right)^{\frac{1}{N'}}  \frac{1}{\ln n\, \lvert x-x_p \rvert^N} w_{1\beta}(x) \, dx\\
        &\le  \int_{\frac{\delta_p}{n} < \lvert x - x_p \rvert < \delta_p} \frac{1}{ \omega_{N-1} \ln n} \frac{1}{\lvert x-x_p \rvert^N} \, dx =1.
    \end{align*}
On the other hand,
\begin{align*}
    \int_{\Omega} e^{\alpha \lvert v_{n,p} \rvert ^{N'}} \, dx &\ge \int_{B\left(x_p, \frac{\delta_p}{n}\right)} e^{(N+1) \ln n} \, dx = \omega_{N-1} n^{N+1} \frac{\delta_p^N}{n^N} \to \infty,
\end{align*}
as $n \to \infty.$ This proves, $ W^{1,N}_{0}(\Omega, w_{1\beta}) \not \hookrightarrow L^{\psi_N}\left(\Omega\right)$.
\par \emph{Next we prove \eqref{moser trudinger weight}}. Because of \Cref{approximation} it is enough to consider $u \in C^{\infty}_{c}(\Omega\setminus\{0\})$ with $\left \lVert \nabla u \right\rVert_{N, w_{1\beta}} \le 1$. We define $v(x) \coloneqq u(x) Y_1^{-\beta/N'}$, for $x\in \Omega$. Then
\begin{align*}
    \nabla v(x) = Y_1^{-\frac{\beta}{N'}} \nabla u(x)  -\frac{\beta}{N'} Y_1^{-\frac{\beta}{N'}+1} u(x) \frac{x}{\lvert x \rvert^2}.
\end{align*}
Hence,
\begin{align*}
    \lvert \lvert  \nabla v \rvert \rvert_{L^N(\Omega)} &\le \left(\int_{\Omega} \lvert \nabla u \rvert^N w_{1\beta}\right)^{1/N} + \frac{\beta}{N'} \left(\int_{\Omega} \frac{\lvert u \rvert ^N}{\lvert x \rvert^N} Y_1^{N - (N-1) \beta} \, dx\right) ^{1/N}\\
    &\stackrel{\eqref{hardy type}}{\le} 1 + \frac{\beta}{1- \beta} = \frac{1}{1-\beta}.
\end{align*}
Therefore, by  \eqref{moser-trudinger} we establish \eqref{moser trudinger weight}. 
\par \emph{Finally, we prove \eqref{equation non radial moser optimal beta} i.e. the optimality of the weight} $Y_1^{-\beta}$. First assume that $a=0.$
 Since $\Omega$ contains the origin so, there exists $\delta>0$ such that $B(0,\delta) \subset \Omega.$ For any $0<r_1<\delta< R_{\Omega}/e$, consider the family of functions given by,
\begin{align*}
        u_{r_1}(r) \coloneqq \begin{cases}
\left[ \left(\ln \frac{e \delta}{r_1}\right)^{1-\beta}-1\right]^{\frac{1}{N'}}, \quad 0 \le r < r_1\\
            \frac{\left(\ln \frac{e \delta}{r}\right)^{1-\beta}-1}{\left[ \left(\ln \frac{e \delta}{r_1}\right)^{1-\beta}-1\right]^{\frac{1}{N}}} , \quad r_1 \le r < \delta \\
            =0, \quad r \ge \delta.
        \end{cases}
    \end{align*}
    By Lemma \ref{contains lip functions}, $u_{r_1} \in  W^{1,N}_{0}\left(\Omega, w_{1\beta}\right)$, and we have
\begin{align*}
    \int_{\Omega} &\lvert \nabla u_{r_1}(x) \rvert^N w_{1\beta} \, dx \notag\\
    &= \frac{\omega_{N-1}(1-\beta)^N}{\left(\ln \frac{e\delta}{r_1}\right)^{1-\beta}-1} \int_{r_1}^{\delta} \left(\ln \frac{e\delta}{r}\right)^{-N \beta} \left(\ln \frac{R_{\Omega}}{r}\right)^{\beta(N-1)} \frac{1}{r} \, dr \notag \\
    &\leq \frac{\omega_{N-1}(1-\beta)^N2^N}{\left(\ln \frac{e\delta}{r_1}\right)^{1-\beta}-1} \int_{r_1}^{\delta} \left(\ln \frac{e\delta}{r}\right)^{-N \beta} \left(\left(\ln \frac{R_{\Omega}}{e\delta}\right)^{\beta(N-1)} + \left(\ln \frac{e\delta}{r}\right)^{\beta(N-1)}\right) \frac{dr}{r} \notag\\
    &\leq  \frac{\omega_{N-1}(1-\beta)^N 2^{N}}{\left(\ln \frac{e\delta}{r_1}\right)^{1-\beta}-1} \int_{r_1}^{\delta} \left(\left(\ln \frac{e\delta}{r}\right)^{- \beta} \left(\ln \frac{R_{\Omega}}{e\delta}\right)^{\beta (N-1)} + \left(\ln \frac{e\delta}{r}\right)^{-\beta}\right) \frac{dr}{r}.
\end{align*}
This implies 
\begin{align}\label{derivation bddness}
 \int_{\Omega} &\lvert \nabla u_{r_1}(x) \rvert^N w_{1\beta} \, dx \leq M \equiv M\left(N, \beta, \Omega\right)>0.
\end{align}
Then clearly, the normalized variant $v_{r_1} = u_{r_1}/M^{1/N}\in W^{1,N}_0 \left (\Omega, w_{1\beta}\right) $ and satisfies $\left\lVert \nabla v_{r_1}\right\rVert_{N, w_{1\beta}} \leq 1.$ Now for any $\alpha>0$, denote $\tilde{\alpha}:= \alpha/M$. Since $f(x)Y_1^{\beta} \to \infty$, as $ x  \to 0$, so taking $r_1$ small enough yields $f(x)Y_{1}^{\beta} \ge (N+1)/\tilde{\alpha} $ on $B(0,r_1).$ Therefore we have
    \begin{align*}
       \frac{1}{\omega_{N-1}} \int_{\Omega} e^{\alpha \lvert v_{r_1} \rvert^{N'} f} &\geq  \int_{0}^{r_1} \exp \left( \Tilde{\alpha} \left(\left(\ln \frac{e\delta}{r_1}\right)^{1-\beta}-1\right)  f(x)\right) r^{N-1} \, dr\\
        &\ge \int_{0}^{r_1} \exp\left((N+1) \left(\left(\ln \frac{e\delta}{r_1}\right)^{1-\beta}-1\right) \left(\ln \frac{R_{\Omega}}{r} \right)^{\beta}\right) r^{N-1} dr\\
        &\geq \frac{1}{N} \exp\left((N+1) \left(\left(\ln \frac{e\delta}{r_1}\right)^{1-\beta}-1\right) \left(\ln \frac{R_{\Omega}}{r_1} \right)^{\beta}- N\ln \frac{1}{r_1}\right).     \end{align*}
clearly, the RHS of the above inequality goes to $\infty$, as $r_1 \to 0.$ This proves \eqref{moser trudinger weight} when $a=0$.

\par For $a\in \overline{\Omega} \setminus \lbrace 0\rbrace$, the proof of \eqref{moser trudinger weight} proceeds analogously to that of \eqref{equation non radial moser optimal} in Theorem \ref{non radial moser}. We therefore omit the details. This concludes the proof of Theorem \ref{moser weighted1}. \\
\end{proof}

\begin{proof} [Proof of \Cref{moser weighted1/2}] 
The proof of Theorem \ref{moser weighted1/2} is a straightforward adaptation of Theorem \ref{moser weighted1}. Hence, the details are omitted  \\
 \end{proof}

\begin{proof} [Proof of \Cref{moser beta>1}]
The proof of \eqref{moser trudger weight 2} proceeds analogously to that of \eqref{moser trudinger weight} in Theorem \ref{moser weighted1}.  

Next, we prove \eqref{moser trudger weight 2 entire space}. Since $\Omega$ contains the origin so there exists $\delta>0$ such that $B(0,\delta) \subset \Omega.$
 For $0<r_1<\delta$, consider the family of functions,
    \begin{align}
        u_{r_1}(r) \coloneqq \begin{cases}
            \left[1 - \left(\ln \frac{e \delta}{r_1}\right)^{1-\beta}\right]^{1/N'} , \quad 0 \le r < r_1\\
          \frac{1 - \left(\ln \frac{e \delta}{r}\right)^{1-\beta}}{\left[1 - \left(\ln \frac{e \delta}{r_1}\right)^{1-\beta}\right]^{1/N}} , \quad r_1 \le r < \delta\\
          0, \quad r\geq \delta.
        \end{cases}
    \end{align}

 By Lemma \ref{contains lip functions}, $u_{r_1} \in  W^{1,N}_{0}\left(\Omega, w_{2\beta}\right)$, and proceeding similarly to the derivation of \eqref{derivation bddness} we conclude that
\begin{align*}
  \int_{\Omega} &\lvert \nabla u_{r_1}(x) \rvert^N w_{2\beta} \, dx=   \int_{\Omega} \lvert\nabla u_{r_1} \rvert^N \left(\ln \frac{e R_{\Omega}}{\lvert x \rvert}\right)^{\beta(N-1)} \, dx \le M\equiv M(N,\beta,\Omega),
\end{align*}
for some positive constant $M$. So, the normalized function $v_{r_1} = u_{r_1}/M^{1/N}\in W^{1,N}_0 \left (\Omega, w_{2\beta}\right) $ and satisfies $\left\lVert \nabla v_{r_1}\right\rVert_{N, w_{2\beta}} \leq 1.$
Now for any $\alpha>0$, we denoting $\tilde{\alpha} = \alpha/M$ and estimate
\begin{align*}
   \int_{\Omega} e^{\alpha \lvert v_{r_1} \rvert ^{N'} \left(\ln \frac{e R_{\Omega}}{|x|}\right)^{\beta}}\, dx &\geq \omega_{N-1} \int_{0}^{\delta} e^{\alpha \lvert v_{r_1} \rvert ^{N'} \left(\ln \frac{e R_{\Omega}}{r}\right)^{\beta}} r^{N-1} \, dr\\
    &\ge \int_{0}^{r_1} \exp \left(\tilde{\alpha} \left[1- \left( \ln \frac{e \delta}{r_1}\right)^{1-\beta}\right]\left( \ln \frac{e R_{\Omega}}{r}\right)^{\beta}\right) r^{N-1} \, dr\\
    &\ge   \exp \left(\tilde{\alpha}\left[ \left( \ln \frac{e \delta}{r_1}\right)^{\beta}- \ln \frac{e \delta}{r_1}\right] - N\ln \frac{1}{r_1} \right) \to \infty,
\end{align*}
as $r_1 \to 0$ since $\beta > 1.$ This proves \eqref{moser trudger weight 2 entire space} completing the proof of Theorem \ref{moser beta>1}. \\
\end{proof}

\section{Proof of \texorpdfstring{\Cref{moser weighted3}}{}}\label{proof moser idea}

In this section we will prove \Cref{moser weighted3}. In view of  Corollary \ref{lattice}, we may assume that $u(r)$ to be non-negative. For $\varepsilon \in (0,1)$ we define,
\begin{align} \label{broken-line}
        \xi_{\varepsilon}(r) \coloneqq \begin{cases}
            \frac{1}{\omega_{N-1}^{\frac{1}{N}}(1-\beta)^{\frac{1}{N'}}} \left(\ln \frac{1}{\varepsilon}\right)^{\frac{1-\beta}{N'}} \quad 0 \le r < \varepsilon\\
            \frac{1}{\omega_{N-1}^{\frac{1}{N}}(1-\beta)^{\frac{1}{N'}}} \left[\frac{\ln \frac{1}{r}}{\left(\ln \frac{1}{\varepsilon}\right)^{1/N}}\right]^{1-\beta} \quad \varepsilon \le r <1.
        \end{cases} 
    \end{align}

First we prove \eqref{equation sharp radial moser beta}, for $\alpha =\alpha_{N,\beta}= N \omega_{N-1}^{\frac{1}{N-1}}(1-\beta) $. By Lemma \ref{contains lip functions}, we conclude that  $\xi_{\varepsilon} \in W^{1,N}_0 \left (\mathbb{B}_N, w_{1\beta}\right)$.  Here, our strategy is to demonstrate that any function $u$ satisfying the hypothesis of \Cref{moser} below, is close enough to functions of the form \eqref{broken-line} and the integrand in \eqref{equation sharp radial moser beta} bounded by some constant that depends on $N$. 

Next we consider the supercritical case i.e. $\alpha >\alpha_{N,\beta}$.  In this case, we will show that the family of functions given in \eqref{broken-line} is sufficient to conclude the result. 
\subsection{Basic Setup in the Critical case : \texorpdfstring{$\alpha = \alpha_{N,\beta}$}{}} \label{critical}
 Le $u\in C_{c, rad}^1 \left(\mathbb{B}_N\right)$ and satisfy $\left \lVert \nabla u \right\rVert^{N}_{N,w_{1\beta}}\le 1$. Then Lemma \ref{radial lemma} $(i)$, implies that 
\begin{align}\label{defi F_u beta}
   F_{u,\beta}(r)\coloneqq \omega_{N-1}(1 - \beta)^{N-1}  Y_{1}(r)^{(1-\beta)(N-1)} u^N(r) \le 1, \text{ for all }r \in (0,1).
\end{align}
Since $u \in C^1_{c, rad}\left(\mathbb{B}_N\right)$ so we have, $F_{u,\beta}(r)= 0$ for $r=0,1.$ Now we set
\begin{align} \label{i2}
    1-\delta &= \max_{r \in [0,1]} F_{u,\beta}(r)=F_{u,\beta}(r_1),
\end{align} 
for some $r_1 \in [0,1]$ and $0\leq \delta \leq 1$ depending on the function $u(r)$. We may assume that $u$ is non trivial, hence $0<r_{1}<1$ and $0\leq \delta<1$. Set
\begin{align} \label{y(r)}
    G_{u,\beta}(r):=  \omega_{N-1}^{\frac{1}{N}} (1-\beta)^{\frac{1}{N'}} \left( Y_1(r_1)\right)^{\frac{(1-\beta)}{N'}}  u(r).
\end{align}
Clearly,
\begin{align} \label{1-delta}
    1-\delta = G_{u,\beta}^{N}(r_1) \quad \text{and} \quad G_{u,\beta}(1) = 0.
\end{align}
Also, as $\left \lVert \nabla u \right\rVert_{N,w_{1\beta}}\le 1$, so we have
\begin{align*} 
    \int_{0}^{1} (1-\beta)^{-N+1} \left( \ln \frac{1}{r_1}\right)^{(1-\beta)(N-1)} \left\lvert G_{u,\beta}'(r) \right\rvert^{N} w_{1\beta}(r) r^{N-1} \, dr \le 1,
\end{align*}
which implies
\begin{align} \label{estimate on y}
     \int_{0}^{1}  \left\lvert G_{u,\beta}'(r) \right\rvert^{N} w_{1\beta}(r) r^{N-1} \, dr \le \left[\frac{1-\beta}{\left( \ln \frac{1}{r_1} \right)^{1-\beta}}\right]^{N-1}.
\end{align}
Now we prove the following crucial lemma.
\begin{lemma}
     Let $G_{u,\beta}$ be as in \eqref{y(r)}, then we have the following inequality
     \begin{align}\label{vv}
     \left(\frac{(1-\beta)G_{u,\beta}(r_1)}{ \left(\ln  \frac{1}{r_1}\right)^{1-\beta}}\right)^{N-2}&\int_{r_1}^{1} \left(  G_{u,\beta}' (r)+ \frac{1-\beta}{r \left(\ln\frac{1}{r}\right)^{\beta}} \frac{G_{u,\beta}(r_1)}{\left( \ln\frac{1}{r_1}\right)^{1-\beta}}\right)^2 Y_1^{-\beta}(r) r \, dr \notag\\
       & \hspace{-1cm}+ \int_{0}^{r_1} \left \lvert G_{u,\beta}'(r)\right\rvert^{N} Y_1^{-\beta(N - 1)}(r) r^{N-1} \, dr \le \frac{\delta (1- \beta)^{N-1}}{\left(\ln  \frac{1}{r_1}\right)^{(1-\beta)(N-1)}}. 
     \end{align}
 \end{lemma}
\begin{proof}
    We have,
    \begin{align}\label{vv first est}
        A &\coloneqq \int_{r_1}^{1} \left(  G_{u,\beta}'(r) + \frac{1-\beta}{r \left(\ln\frac{1}{r}\right)^{\beta}} \frac{G_{u, \beta}(r_1)}{\left( \ln\frac{1}{r_1}\right)^{1-\beta}}\right)^2 Y_1^{-\beta}(r) r \, dr \notag \\
        & \begin{aligned}
        =\int_{r_1}^1 \left(G_{u,\beta}'(r)\right)^2 Y_1^{-\beta}(r) r \, dr &+ 2(1-\beta) \frac{G_{u,\beta}(r_1)}{\left( \ln \frac{1}{r_1}\right)^{1-\beta}} \int_{r_1}^1 G_{u,\beta}'(r) \, dr \notag\\
        & + \frac{G_{u,\beta}^2(r_1)(1-\beta)^2}{\left( \ln \frac{1}{r_1}\right)^{2(1-\beta)}}\int_{r_1}^1 \frac{Y_1^{\beta}(r) dr}{r }
        \end{aligned}\notag \\
        &\stackrel{\eqref{1-delta}}{=}  \int_{r_1}^1 \left(G_{u,\beta}'(r)\right)^2 Y_1^{-\beta}(r) r \, dr - 2(1-\beta) \frac{G_{u,\beta}^2(r_1)}{\left( \ln \frac{1}{r_1}\right)^{1-\beta}} + (1-\beta) \frac{G_{u,\beta}^2(r_1)}{\left( \ln \frac{1}{r_1}\right)^{1-\beta}} \notag \\
        &= \int_{r_1}^1 \left(G_{u,\beta}(r)'\right)^2 Y_1^{-\beta}(r) r \, dr -  \frac{(1-\beta)G_{u,\beta}^2(r_1)}{\left( \ln \frac{1}{r_1}\right)^{1-\beta}}.
    \end{align}
 If $N=2$, then the proof of \eqref{vv} follows from by combining \eqref{vv first est} with \eqref{estimate on y} and then using \eqref{1-delta}. So, for the rest of the proof we assume $N>2$. Using \eqref{1-delta} and H\"older inequality we estimate
\begin{align*}
    G_{u,\beta}(r_1) &\le  \int_{r_1}^{1} \left \lvert G_{u,\beta}'(s) \right \rvert \left(\ln \frac{1}{s} \right)^{\beta/N'} s^{1/N'} \left(\ln \frac{1}{s} \right)^{-\beta/N'} s^{-1/N'} \, ds\\
    & \le \left(\int_{r_1}^{1} \left \lvert G_{u,\beta}'(s) \right\rvert^N \left(\ln \frac{1}{s} \right)^{\beta(N-1)} s^{N-1} \, ds \right)^{1/N} \left(\int_{r_1}^{1} \frac{ds}{s \left(\ln \frac{1}{s}\right)^{\beta}}\right)^{1/N'}.
\end{align*}
Thus we have,
\begin{align} \label{y^N}
    \frac{(1-\beta)^{N-1} G_{u,\beta}^N(r_1)}{\left( \ln \frac{1}{r_1}\right)^{(1-\beta)(N-1)}} \le \int_{r_1}^{1} \left \lvert G_{u,\beta}'(r) \right\rvert^N \left(\ln \frac{1}{r} \right)^{\beta(N-1)} r^{N-1} \, dr .
\end{align}
We use the temporary shorthand $\mathcal{A}_{r_1} \coloneqq  (1-\beta)G_{u,\beta}(r_1)/{ \left(\ln  1/r_1\right)^{1-\beta}}$ and estimate using H\"older inequality with the exponent $N/2$ and $N/(N-2)$
\begin{align*}
    \int_{r_1}^1 \mathcal{A}_{r_1}^{N-2} &\left(G^{\prime}_{u,\beta}(r)\right)^2 Y_1^{-\beta}(r)  r \, dr \\ &= \int_{r_1}^1  \left(G^{\prime}_{u,\beta}(r)\right)^2 \left(Y_1(r)\right)^{\frac{-2\beta(N-1)}{N}} r^{\frac{2(N-1)}{N}} \mathcal{A}_{r_1}^{N-2}  \left(Y_1(r)\right)^{\frac{\beta(N-2)}{N}} r^{-\frac{(N-2)}{N}} \, dr\\
    & \le \left(\int_{r_1}^{1} \lvert G^{\prime}_{u,\beta}(r) \rvert^N \left(\ln \frac{1}{r} \right)^{\beta(N-1)} r^{N-1}  \, dr\right)^{2/N}  \left(A_{r_1}^N \int_{r_1}^{1} \frac{dr}{r \left(\ln \frac{1}{r}\right)^{\beta}}\right)^{\frac{N-2}{N}}\\
    &\le \left[\frac{(1-\beta)^{N-1}G_{u,\beta}^N(r_1)}{\left(\ln \frac{1}{r_1}\right)^{(1-\beta)(N-1)}}\right]^{\frac{N-2}{N}}  \left(\int_{r_1}^{1} \lvert G^{\prime}_{u,\beta}(r) \rvert^N \left(\ln \frac{1}{r} \right)^{\beta(N-1)} r^{N-1} \, dr\right)^{\frac{2}{N}} \\
    &\stackrel{\eqref{y^N}}{\le}\int_{r_1}^{1} \lvert G^{\prime}_{u,\beta}(r) \rvert^N  \left(\ln \frac{1}{r} \right)^{\beta(N-1)} r^{N-1} \, dr.
\end{align*}
Combining this with \eqref{vv first est}, we estimate the left hand side of \eqref{vv} as 
\begin{align*}
 \mbox{L.H.S. of } \eqref{vv} &\leq  \int_{r_1}^{1}  \lvert G^{\prime}_{u,\beta}(r) \rvert^N \left(\ln \frac{1}{r} \right)^{\beta(N-1)} r^{N-1} \, dr -  \frac{(1-\beta)^{N-1} G_{u,\beta}^N(r_1)}{\left( \ln \frac{1}{r_1}\right)^{(1-\beta)(N-1)}}\\
    &\hspace{3.5cm}+  \int^{r_1}_{0} \lvert G^{\prime}_{u,\beta}(r) \rvert^N \left(\ln \frac{1}{r} \right)^{\beta(N-1)} r^{N-1} \, dr\\
    &=  \int_{0}^{1} \lvert G^{\prime}_{u,\beta}(r) \rvert^N  \left(\ln \frac{1}{r} \right)^{\beta(N-1)} r^{N-1} \, dr - \frac{(1-\beta)^{N-1} G_{u,\beta}^N(r_1)}{\left( \ln \frac{1}{r_1}\right)^{(1-\beta)(N-1)}}\\
    &\stackrel{\eqref{estimate on y}}{=} \left(1-G_{u,\beta}^N(r_1)\right) \frac{(1-\beta)^{N-1} }{\left( \ln \frac{1}{r_1}\right)^{(1-\beta)(N-1)}}\stackrel{\eqref{1-delta}}{=} \frac{\delta(1-\beta)^{N-1} }{\left( \ln \frac{1}{r_1}\right)^{(1-\beta)(N-1)}}.
 \end{align*}
 This completes the proof. \\
\end{proof}

\begin{corollary} \label{delta=0}
   Let $\delta$ be as defined in \eqref{i2}. Then  $\delta >0$.
\end{corollary}
\begin{proof}
     If $\delta= 0$, then \eqref{vv} and \eqref{1-delta} implies $G_{u,\beta}(r) =1$, for any $r\in [0,r_1]$ and  $G_{u,\beta}$ solves the following ODE 
\begin{align*}
    y^{\prime}(r) = - \frac{1-\beta}{r \left(\ln\frac{1}{r}\right)^{\beta}\left( \ln\frac{1}{r_1}\right)^{1-\beta}}, \mbox{ in } (r_1,1),  \mbox{ and } y(1) = 0.
\end{align*}
But this implies $G_{u,\beta}(r) = \left(\ln \frac{1}{r}/\ln \frac{1}{r_1}\right)^{1-\beta}$, for any $r\in (r_1,1]$. Thus, it follows from \eqref{y(r)} that $u(r)=\xi_{r_1}(r)$. This contradicts the fact that $u\in  C_{c, rad}^1 \left(\mathbb{B}_N\right)$.
\end{proof}
Now we are ready to prove \eqref{equation sharp radial moser beta} in  \Cref{moser weighted3}. We first derive three crucial point wise estimates, which are necessary in proving Proposition \ref{second battle eq} below. We note the resemblance of these estimates with estimates $(10)$ and $(11)$ in \cite{Moser}.
\begin{lemma} \label{moser first estimate}
    Let $G_{u,\beta}$ and $r_1$ be defined as in \eqref{y(r)} and \eqref{1-delta}, respectively. Then, for all $r \in [0,r_1]$, we have
    \begin{align}\label{moser first estimate eq}
        G_{u,\beta}(r) \le 1 + \delta^{\frac{1}{N}} \left(\left(\frac{\ln \frac{1}{r}}{\ln \frac{1}{r_1}}\right)^{1-\beta}-1\right)^{\frac{1}{N'}}.
    \end{align}
\end{lemma}
\begin{proof}
    We estimate using H\"older's inequality and \eqref{1-delta}
\begin{align*}
    G_{u,\beta}(r) 
    &\le G_{u,\beta}(r_1)+  \int^{r_1}_{r} \left\lvert G^{\prime}_{u,\beta}(s) \right\rvert \left(\ln \frac{1}{s} \right)^{\frac{\beta}{N^{\prime}}} s^{\frac{1}{N^{\prime}}} \left(\ln \frac{1}{s} \right)^{-\frac{\beta}{N^{\prime}}} s^{-\frac{1}{N^{\prime}}} \, ds\\
    & \le 1+\left(\int^{r_1}_{r} \left\lvert G^{\prime}_{u,\beta}(s) \right\rvert^N \left(\ln \frac{1}{s} \right)^{\beta(N-1)} s^{N-1} \, ds \right)^{\frac{1}{N}} \left(\int^{r_1}_{r} \frac{ds}{s \left(\ln \frac{1}{s}\right)^{\beta}}\right)^{\frac{1}{N^{\prime}}}\\
    & \stackrel{\eqref{vv}}{\le} 1+ \frac{\delta^{\frac{1}{N}} (1-\beta)^{\frac{1}{N'}}}{\left(\ln  \frac{1}{r_1}\right)^{\frac{1-\beta}{N'}}} \left(\frac{\left(\ln\frac{1}{r}\right)^{1-\beta}- \left(\ln\frac{1}{r_1}\right)^{1-\beta}}{1-\beta}\right)^{\frac{1}{N'}}
\end{align*}
This completes the proof. \\
\end{proof}

\begin{lemma} \label{moser second estimate}
      Let $G_{u,\beta}$ and $r_1$ be defined as in \eqref{y(r)} and \eqref{1-delta}, respectively. Then, for all $r \in [r_1,1]$, we have
        \begin{align}\label{moser second estimate eq}
        G_{u,\beta}(r) \le \left(\frac{ \ln\frac{1}{r}}{ \ln\frac{1}{r_1}}\right)^{1-\beta} + \delta^{\frac{1}{2}} \left(1- \left(\frac{\ln \frac{1}{r}}{\ln\frac{1}{r_1}}\right)^{1-\beta}\right)^{\frac{1}{2}} G_{u,\beta}^{-\frac{N-2}{2}}(r_1).
    \end{align}
\end{lemma}
\begin{proof}
    We will again use \eqref{vv} in order to estimate $G_{u,\beta}$ in $[r_1,1]$. We have
\begin{align*}
    G_{u,\beta}(r)&=  G_{u,\beta}(r_1) \\ 
    & \hspace{1cm}+\int_{r_1}^{r} \left(  G_{u,\beta}^{\prime}(s) + \frac{(1-\beta)G_{u,\beta}(r_1)}{s \left(\ln\frac{1}{s}\right)^{\beta}\left( \ln\frac{1}{r_1}\right)^{1-\beta}} - \frac{(1-\beta)G_{u,\beta}(r_1)}{s \left(\ln\frac{1}{s}\right)^{\beta}\left( \ln\frac{1}{r_1}\right)^{1-\beta}}\right) \, ds\\
    &\begin{aligned}
      = G_{u,\beta}(r_1)+\int_{r_1}^{r} \frac{Y_1^{\frac{\beta}{2}}(s)}{\sqrt{s}}\left(  G_{u,\beta}^{\prime}(s) + \frac{(1-\beta)G_{u,\beta}(r_1)}{s \left(\ln\frac{1}{s}\right)^{\beta}\left( \ln\frac{1}{r_1}\right)^{1-\beta}} \right) &Y_1^{-\frac{\beta}{2}}(s) \sqrt{s} \, ds\\
    &\hspace{-2.5cm} - \int_{r_1}^{r} \frac{(1-\beta)G_{u,\beta}(r_1)}{s \left(\ln\frac{1}{s}\right)^{\beta}\left( \ln\frac{1}{r_1}\right)^{1-\beta}} \, ds.
    \end{aligned} 
    \end{align*}
    We now use Cauchy–Schwarz inequality, \eqref{vv}  and 
    $$\int_{r_1}^{r} \frac{(1-\beta)G_{u,\beta}(r_1)}{s \left(\ln\frac{1}{s}\right)^{\beta}\left( \ln\frac{1}{r_1}\right)^{1-\beta}} \, ds = G_{u,\beta}(r_1) \left(1 - \left(\frac{\ln \frac{1}{r}}{\ln \frac{1}{r_1}}\right)^{1-\beta}\right),
$$
to derive the estimate
    \begin{align*}
      G_{u,\beta}(r)
    &\le \left(\int_{r_1}^{r}\frac{ds}{s \left(\ln \frac{1}{s}\right)^{\beta}}\right)^{\frac{1}{2}}\left(\frac{\delta(1-\beta)}{  \left(\ln \frac{1}{r_1}\right)^{1-\beta}} \right)^{\frac{1}{2}} G_{u,\beta}^{-\frac{N-2}{2}}(r_1) + G_{u,\beta}(r_1) \left(\frac{\ln \frac{1}{r}}{\ln \frac{1}{r_1}}\right)^{1-\beta}\\
      &= G_{u,\beta}(r_1) \left(\frac{\ln \frac{1}{r}}{\ln \frac{1}{r_1}}\right)^{1-\beta}+\delta^{1/2} \left(1-\left(\frac{\ln \frac{1}{r}}{\ln \frac{1}{r_1}}\right)^{1-\beta}\right)^{1/2} G_{u,\beta}^{-\frac{N-2}{2}}(r_1).
\end{align*}
This proves \eqref{moser second estimate eq}. \\
\end{proof}
\begin{lemma} \label{moser third estimate}
Let $G_{u,\beta}$, $\delta$ and $r_1$ be defined as in \eqref{y(r)}, \eqref{i2} and \eqref{1-delta}, respectively. Suppose, $0<\delta<1/2$. Then, for all $r \in [r_1,1]$, we have        
\begin{align}\label{moser third estimate eq}
        G_{u,\beta}(r) \le \left(\frac{\ln \frac{1}{r}}{\ln \frac{1}{r_1}}\right)^{1-\beta} + (2\delta)^{1/N} \left(\frac{\ln \frac{1}{r}}{\ln \frac{1}{r_1}}\right)^{\frac{1-\beta}{N'}}.
        \end{align}
\end{lemma}
\begin{proof}
    We may assume, $r_1< r< 1.$ Using $G_{u,\beta}(1)=0$, we write 
\begin{align*}
    G_{u,\beta}(r)  &= - \int_{r}^1  G_{u,\beta}^{\prime}(s) \left(\ln \frac{1}{s} \right)^{\frac{\beta}{N'}} s^{\frac{1}{N'}} \left(\ln \frac{1}{s} \right)^{-\frac{\beta}{N'}} s^{-\frac{1}{N'}} \, ds\\
    & \le \left(\int_{r}^{1} \left\lvert G_{u,\beta}^{\prime}(s) \right\rvert^N \left(\ln \frac{1}{s} \right)^{\beta (N-1)} s^{N-1} \, ds \right)^{\frac{1}{N}} \left(\int_{r}^{1} \frac{ds}{s \left(\ln \frac{1}{s}\right)^{\beta}}\right)^{\frac{1}{N'}}\\
    &\le \frac{1}{(1-\beta)^{\frac{1}{N'}}}\left( \ln \frac{1}{r}\right)^{\frac{1-\beta}{N'}} \left(\int_{r}^{1} \left\lvert G_{u,\beta}'(s) \right\rvert^N w_{1\beta}(s) s^{N-1} \, ds \right)^{\frac{1}{N}}.
    \end{align*}
    This implies that,
    \begin{align} \label{max}
    \frac{(1-\beta)^{N-1} G^N_{u,\beta}(r)}{\left( \ln\frac{1}{r}\right)^{(1-\beta)(N-1)}} \le \int_{r}^{1} \left\lvert G_{u,\beta}'(s) \right\rvert^N w_{1\beta}(s) s^{N-1} \, ds .
\end{align}
Also, 
\begin{align*}
    \left\lvert G_{u,\beta}(r) - G_{u,\beta}(r_1) \right\rvert^N &= \left \lvert - \int_{r_1}^{r} G_{u,\beta}'(s) \left(\ln \frac{1}{s} \right)^{\frac{\beta}{N'}} s^{\frac{1}{N'}} \left(\ln \frac{1}{s} \right)^{-\frac{\beta}{N'}} s^{-\frac{1}{N'}} \, ds \right \rvert^N\\
    &\le \left(\frac{\left(\ln\frac{1}{r_1}\right)^{1-\beta}-\left(\ln\frac{1}{r}\right)^{1-\beta}}{1-\beta}\right)^{N-1} \int_{r_1}^{r}\left \lvert G_{u,\beta}'(s) \right\rvert^N w_{1\beta} s^{N-1}.
    \end{align*}
    This implies that,
    \begin{align} \label{min}
    \frac{(1-\beta)^{N-1}\left\lvert(G_{u,\beta}(r)-G_{u,\beta}(r_1)\right\rvert^N}{\left(\left(\ln\frac{1}{r_1}\right)^{1-\beta}-\left(\ln\frac{1}{r}\right)^{1-\beta}\right)^{N-1}} &\le \int_{r_1}^{r} \left\lvert G_{u,\beta}'(s) \right\rvert^N w_{1\beta}(s) s^{N-1} \, ds.
\end{align}

Adding \eqref{max}, \eqref{min} and using \eqref{estimate on y} we have
\begin{align*}
  \frac{(1-\beta)^{N-1}G^N_{u,\beta}(r)}{\left( \ln\frac{1}{r}\right)^{(1-\beta)(N-1)}}  &+  \frac{(1-\beta)^{N-1}\left \lvert G_{u,\beta}(r)-G_{u,\beta}(r_1)\right\rvert^N}{\left(\left(\ln\frac{1}{r_1}\right)^{1-\beta}-\left(\ln\frac{1}{r}\right)^{1-\beta}\right)^{N-1}}\leq \frac{ (1- \beta)^{N-1}}{\left(\ln  \frac{1}{r_1}\right)^{(1-\beta)(N-1)}},
\end{align*}
which implies
\begin{align} \label{root}
    \frac{G^N_{u,\beta}(r)}{\left(\frac{\ln \frac{1}{r}}{ \ln\frac{1}{r_1}}\right)^{(1-\beta)(N-1)}} + \frac{\left\lvert G_{u,\beta}(r)-G_{u,\beta}(r_1) \right\rvert^N}{\left(1-\left(\frac{ \ln \frac{1}{r}}{\ln \frac{1}{r_1}}\right)^{1-\beta}\right)^{N-1}} \leq 1.
\end{align}
For fixed $r$ and $r_{1}$ with $r_1< r <1$ we use the following temporary shorthands.  
 $$\tau \coloneqq \left(\frac{ \ln \frac{1}{r}}{\ln \frac{1}{r_1}}\right)^{1-\beta}, \mbox{ and  } \Gamma \coloneqq G_{u,\beta}(r)$$  
 So, \eqref{root} takes the form,
\begin{align} \label{z}
   \left( \frac{\Gamma}{\tau}\right)^N \tau+ \left \lvert\frac{\Gamma-G_{u,\beta}(r_1)}{1-\tau}\right\rvert^N(1-\tau) \le 1.
\end{align}
Clearly, $0< \tau<1$. Let $\eta$ be such that $\Gamma=(\tau+\eta)G_{u,\beta}(r_1)$. Now, if $\eta \le 0$, then we have $\Gamma/G_{u,\beta}(r_1) = \tau+\eta \le \tau $ and this implies 
\begin{align}\label{case1 convexity}
G_{u,\beta} (r) \le \left(\frac{ \ln \frac{1}{r}}{\ln \frac{1}{r_1}}\right)^{1-\beta}.
\end{align} 
Next, we assume $\eta>0$. Substituting $\Gamma=(\tau+\eta)G_{u,\beta}(r_1)$ into \eqref{z} gives,
\[
\left(1 + \frac{\eta}{\tau}\right)^N\tau + \left \lvert 1 - \frac{\eta}{1-\tau} \right \rvert ^N (1-\tau) \le \frac{1}{G_{u,\beta}^N(r_1)} = \frac{1}{1- \delta}.
\]
Now, note that $\lvert 1- a \rvert^p \ge 1- pa$, for all $a\in \mathbb{R}$.  Using this in the above estimate, we deduce,
\begin{align*}
    \left(1+ N \frac{\eta}{\tau} + \left(\frac{\eta}{\tau}\right)^N\right)\tau + \left(1- N \frac{\eta}{1-\tau} \right)(1-\tau) \le \frac{1}{1- \delta},
    \end{align*}
which implies $\eta^N/\tau^{N-1} \le \delta/(1-\delta) < 2 \delta.$
Thus, we have 
\begin{align}\label{case2 convexity}
G_{u,\beta}(r)= \Gamma=G_{u,\beta}(r_1)\left(\tau+\eta\right) 
&\le \tau+\eta \le \tau + \left(2 \delta\right)^{\frac{1}{N}} \tau^{\frac{N-1}{N}} \notag \\
&= \left(\frac{ \ln \frac{1}{r}}{\ln \frac{1}{r_1}}\right)^{1-\beta} + \left(2 \delta\right)^{\frac{1}{N}} \left(\frac{ \ln \frac{1}{r}}{\ln \frac{1}{r_1}}\right)^{\frac{1-\beta}{N'}}.
\end{align}
Therefore, combining \eqref{case1 convexity} and \eqref{case1 convexity} we establish \eqref{moser third estimate eq}. \\
\end{proof}

Finally we are ready to prove Theorem \eqref{moser weighted3}. for small enough $\delta.$ 
\begin{proposition} \label{second battle}
  There exists $\delta_0 \equiv \delta_0(N)>0$, such that if  $ \delta$ given by \eqref{i2} satisfy $0<\delta \le \delta_{0}$, then
     \begin{align}\label{second battle eq}
        \omega_{N-1} \int_{0}^{1} e^{\alpha_{N,\beta} u^{N'}(r) X_1^{-\beta}(r)} r^{N-1} \, dr\le c_{N,\beta}, 
     \end{align}
      for some constant $c_{N,\beta}>0$ depending only on $N$ and $\beta$.
 \end{proposition}
\begin{proof}
    Throughout the proof, $c_{N,\beta}$ denotes a generic positive constant depending only on $N$ and $\beta$. Additionally, we will use the following elementary inequality.  
\begin{align} \label{elementary inequality}
    (a+b)^p \le a^p + p 2^{p-1}\left(a^{p-1}b + b^p\right) \mbox{ for any } a, b \geq 0 \mbox{ and all } p\ge 1.
\end{align} 
Assume $0<\delta<1/(2d)$, where $d\equiv d_N>1$ will be chosen later. Let $r_1$ be as defined in \eqref{1-delta}. We consider the partition of $[0,1]$ given by $\lbrace 0, r^{a_1}_1, r^{a_2}_1, r_1^{a_3}, 1 \rbrace$, where 
     \begin{align}\label{tilde s def}
       a_1 \coloneqq {\left(1+d \delta\right)^{\frac{1}{1-\beta}}}, \quad  a_2 \coloneqq {\left(1-d \delta\right)^{\frac{1}{1-\beta}}}, \quad a_3\coloneqq  {\left(\frac{1}{2N'}\right)^{\frac{1}{1-\beta}}},
     \end{align}
 and define 
 \begin{align*}
 \Delta_1\coloneqq [0, r^{a_1}_1], \quad \Delta_2 \coloneqq [ r^{a_1}_1, r^{a_2}_1],  \quad \Delta_3 \coloneqq [ r^{a_2}_1, r^{a_3}_1], \quad \Delta_4 \coloneqq [ r^{a_3}_1,1].
 \end{align*}
We will estimate the integrand in \eqref{second battle eq} separately for each region $\Delta_i$, where $i=1,\dots,4$.
\par \emph{ \text{Estimate in the region} $\Delta_1$:} Clearly $\Delta_1 \subset [0,r_1]$. Let $r\in \Delta_1$. We use the temporary shorthand $t \coloneqq \left(\ln \frac{1}{r}\big/\ln \frac{1}{r_1}\right)^{1-\beta} .$ Then using \eqref{moser first estimate eq} and \eqref{elementary inequality} we have,
\begin{align*}
    G^{N'}_{u,\beta}(r) & \le \left(1 + \delta^{\frac{1}{N}} \left( t -1 \right)^{\frac{1}{N'}}\right)^{N'} \le 1 + N'2^{N'-1} \left(\delta^{\frac{1}{N}} \left( t -1 \right)^{\frac{1}{N'}}+  \delta^{\frac{N'}{N}} (t -1)\right).
    \end{align*}
    Since $\delta<1$, this implies,
    \begin{align} \label{i1}
    G^{N'}_{u,\beta}(r) - t & \le \left(t-1\right) \left[-1 +N'2^{N'-1} \left(\frac{\delta}{t-1}\right)^{\frac{1}{N}}+ N'2^{N'-1}\delta^{\frac{1}{N}} \right].
\end{align}
Now, as  $r \in  \Delta_1=  \left[0,\tilde{s}_1\right]$, so by the definition in \eqref{tilde s def} we obtain
\[
\frac{\delta}{t-1} \le \frac{1}{d}.
\]
Using this and the bound $\delta < 1/(2d)$, we derive from \eqref{i1} that
\begin{align}\label{d choice}
     G_{u,\beta}^{N'}(r) - t & \le (t-1)\left(-1 + \frac{N'2^{N'-1}}{d^{\frac{1}{N}}} + \frac{N'2^{N'-1}}{\left(2 d \right)^{\frac{1}{N}}}\right).
\end{align}
We now make the following choice for d, which will be useful for our estimates in regions $\Delta_3$ and $\Delta_4$, in addition to the current region.
\begin{equation}\label{d exact}
\begin{aligned}
  d =d_{N}\coloneqq \max \Bigg\{ 
  &\left( \frac{NN'2^{N'}}{2N-1}\left(1+\frac{1}{2^{\frac{1}{N}}}\right)\right)^N,\left(NN'2^{2N'+2}\right)^2, \\
  &\hspace{4cm} \left(\frac{NN'2^{N'+1}}{\left(2-2^{\frac{N-2}{N-1}}\right)N-1}\right)^N 
  \Bigg\}.
\end{aligned}
\end{equation}
With this choice of $d>1$ we derive from \eqref{d choice}
\begin{align*}
G_{u,\beta}^{N'}(r) - t \leq (t-1)\left(-1 + \frac{N'2^{N'-1}}{d^{\frac{1}{N}}}\left(1 + \frac{1}{ 2^{\frac{1}{N}}}\right)\right) = \frac{ 1-\left(\frac{\ln \frac{1}{r}}{\ln \frac{1}{r_1}}\right)^{1-\beta} }{2N}.
\end{align*} 
By the above estimate together with \eqref{y(r)}, the Young's inequality and the fact that $t>1$ we estimate
\begin{align*}
   \alpha_{N,\beta} \frac{u^{N'}(r)} {X_1^{\beta}(r)}
    &= N G_{u,\beta}^{N'}(r) \left( \ln \frac{1}{r_1} \right)^{1-\beta} \left(\ln \frac{e}{r}\right)^{\beta} \\
    & =N \left(G_{u,\beta}^{N'}(r)-t\right)\left( \ln \frac{1}{r_1} \right)^{1-\beta} \left(\ln \frac{e}{r}\right)^{\beta}  + N \left( \ln \frac{1}{r} \right)^{1-\beta}  \left(\ln \frac{e}{r}\right)^{\beta}\\
    &\le \left(N - \frac{1}{2} \right) \left( \ln \frac{1}{r} \right)^{1-\beta} \left(\ln \frac{e}{r}\right)^{\beta}+ \frac{1}{2} \left( \ln \frac{1}{r_1} \right)^{1-\beta} \left(\ln \frac{e}{r}\right)^{\beta}\\
     &\le \left(N - \frac{1}{2} \right) \left[(1-\beta)\ln \frac{1}{r} + \beta \ln \frac{e}{r}\right]+ \frac{1}{2}  \left[(1-\beta)\ln \frac{1}{r_1} + \beta \ln \frac{e}{r}\right]\\
     &=N \beta \ln\frac{e}{r}+ \left(N - \frac{1}{2}\right)\left(1-\beta \right) \ln \frac{1}{r} + \frac{1-\beta}{2} \ln \frac{1}{r_1}\\ 
     &=N \beta + \left(N - \frac{1-\beta}{2}\right)\ln \frac{1}{r} + \frac{1-\beta}{2} \ln \frac{1}{r_1}.
\end{align*}
Therefore, integrating over $\Delta_1,$ we have
\begin{align}\label{bound over Delta1}
    \int_{\Delta_1}  e^{\alpha_{N,\beta} \frac{u^{N'}(r)} {X_1^{\beta}(r)}} r^{N-1} \, dr& \le e^{N \beta} \int_{0}^{r_1} e^{\left(N - \frac{1-\beta}{2}\right) \ln \frac{1}{r} + \frac{1-\beta}{2} \ln \frac{1}{r_1}} r^{N-1} \, dr \notag \\
    &= e^{N \beta} r_1^{\frac{\beta-1}{2}} \int_{0}^{r_1} r^{\frac{1-\beta}{2} -1} \, dr \le \frac{2 e^{N \beta}}{1- \beta}.
\end{align}

\par \emph{ \text{Estimate} in the region $\Delta_2$}: Let $r\in \Delta_2$.  From \eqref{i2} and the definition in \eqref{defi F_u beta} we have,
\begin{align*}
     \omega_{N-1}(1 - \beta)^{N-1}  u^N(r)  \left( \ln \frac{1}{r}\right)^{(\beta-1)(N-1)} \le 1-\delta.
\end{align*}
It follows that,
\begin{align*}
    N\omega_{N-1}^{\frac{1}{N-1}}(1-\beta)\left(\ln \frac{e}{r}\right)^{\beta} u^{N'}(r) &\le N (1-\delta)^{\frac{1}{N-1}} \left(\ln \frac{1}{r}\right)^{1-\beta} \left(\ln \frac{e}{r}\right)^{\beta}\\
    &\le N (1-\delta)^{\frac{1}{N-1}} \left((1-\beta) \ln \frac{1}{r} + \beta \ln \frac{e}{r}\right)\\
    &\le N \beta + N (1-\delta)^{\frac{1}{N-1}} \ln \frac{1}{r}.
\end{align*}
Using this, alongside the change of variables $z = \ln(1/r)$, and letting $z_1 = \ln (1/r_1)$, we deduce
\begin{align}\label{pre bound over Delta2}
    \int_{\Delta_2} e^{ \alpha_{N,\beta} \frac{u^{N'}(r)}{ X_1^{\beta}(r)}} r^{N-1} \, dr
    &\le e^{N \beta}\int_{\Delta_2} e^{N (1-\delta)^{\frac{1}{N-1}} \ln \frac{1}{r}} r^{N-1} \, dr \notag\\
    &\le e^{N \beta}\int_{a_1z_1}^{a_2 z_1} e^{N (1-\delta)^{\frac{1}{N-1}} z -Nz} \, dz \notag \\
    &\le e^{N \beta} \int_{a_1 z_1}^{a_2 z_1} e^{ -\frac{N \delta}{N-1}z} \, dz \notag \\
    &\stackrel{\eqref{tilde s def}}{\le} e^{N \beta} \left[(1+ d \delta)^{\frac{1}{1-\beta}} - (1- d \delta)^{\frac{1}{1-\beta}}\right] z_1 e^{-N'\delta (1- d \delta)^{\frac{1}{1-\beta}}z_1}, 
\end{align}
where to derive the second last inequality, we have used the fact that, $\left(1-\delta\right)^{1/(N-1)}$ is bounded from above by $1 -\delta/(N-1)$ for $N\ge 2.$
Now, using the  Mean Value Theorem, the choice of $d$ in \eqref{d exact}, and the fact that $0<\delta<1/(2d)$, we conclude that there exists a constant $c_{N,\beta}>0$  such that 
\begin{align*}
    (1+ d \delta)^{1/(1-\beta)} - (1- d \delta)^{1/(1-\beta)} \le \delta  c_{N,\beta}.
\end{align*}
Therefore, first using this estimate in \eqref{pre bound over Delta2}, and then using the fact that $xe^{-x} \le e^{-1}$ for all $x \ge 0$, we establish 
\begin{align}\label{bound over Delta2}
    \int_{\Delta_2} e^{ \alpha_{N,\beta} \frac{u^{N'}(r)}{ X_1^{\beta}(r)}} r^{N-1} \, dr \le c_{N,\beta}  \delta z_1  e^{-N'\delta z_12^{-\frac{1}{1-\beta}}} \le c_{N,\beta}.
\end{align}

\par \emph{ \text{Estimate} in the region $\Delta_3$}: Let $r\in \Delta_3 \subset [r_1,1] $. Denote $t \coloneqq \left(\frac{ \ln \frac{1}{r}}{\ln \frac{1}{r_1}}\right)^{1-\beta} $. Then clearly we have 
\begin{align} \label{t}
    \frac{1}{2N'} \le t \le 1- d \delta < 1.
\end{align} 
Thus we have 
    \begin{align} \label{t'}
         (N'-1)(1-t) \leq 1-t^{N'-1}, \mbox{ and }  \delta^{1/2} \le d^{-1/2}(1-t)^{1/2}.
    \end{align}
Now from \eqref{moser second estimate eq} we have,
\[ G_{u,\beta}(r) \le t + \delta^{1/2}(1-t)^{1/2}G_{u,\beta}^{-\frac{N-2}{2}}(r_1).
\]
Combining this with the bound $G_{u,\beta}^{-\frac{N-2}{2}}(r_1) \le 2^{\frac{N-2}{N}}\leq 2$, which follows from \eqref{1-delta} and the fact that $0 < \delta< 1/(2d)<1/2$, we estimate
\begin{align*}
 G_{u,\beta}^{N'}(r) -t & \le \left(t + 2\delta^{1/2}(1-t)^{1/2}\right)^{N'} - t\\
    &\stackrel{\eqref{elementary inequality}}{\le} t^{N'} -t + N'2^{2N'-1}\left( t^{N'-1} \delta^{\frac{1}{2}}(1-t)^{\frac{1}{2}} + \delta^{\frac{N'}{2}}(1-t)^{\frac{N'}{2}}\right)\\
    &\le  -t(1 - t^{N'-1}) +N'2^{2N'} \delta^{\frac{1}{2}}(1-t)^{\frac{1}{2}}\\
    &\stackrel{\eqref{t}, \eqref{t'}}{\le} -\frac{N'-1}{2N'}(1-t) + \frac{N'2^{2N'}}{d^{\frac{1}{2}}}(1-t)\stackrel{\eqref{d exact}}{=} - \frac{1}{4N}(1-t).
    \end{align*}
By the above estimate together with \eqref{y(r)}, we derive
\begin{align*}
   \alpha_{N,\beta} \frac{u^{N'}(r)} {X_1^{\beta}(r)}
    &= N G_{u,\beta}^{N'}(r) \left( \ln \frac{1}{r_1} \right)^{1-\beta} \left(\ln \frac{e}{r}\right)^{\beta} \\ 
    & =N \left(G_{u,\beta}^{N'}(r)-t\right)\left( \ln \frac{1}{r_1} \right)^{1-\beta} \left(\ln \frac{e}{r}\right)^{\beta}  + N \left( \ln \frac{1}{r} \right)^{1-\beta}  \left(\ln \frac{e}{r}\right)^{\beta}\\
    &\le \left(N + \frac{1}{4}\right) \left(\ln \frac{e}{r}\right)^{\beta} \left(\ln \frac{1}{r}\right)^{1-\beta}- \frac{1}{4}  \left(\ln \frac{e}{r}\right)^{\beta} \left(\ln \frac{1}{r_1}\right)^{1-\beta}\\
     &\le \left(N + \frac{1}{4}\right) \left(\beta \ln \frac{e}{r} + (1-\beta) \ln \frac{1}{r}\right)-\frac{1}{4}  \left(\ln \frac{e}{r}\right)^{\beta} \left(\ln \frac{1}{r_1}\right)^{1-\beta}\\
     &= \left(N + \frac{1}{4}\right)  \beta + \left(N + \frac{1}{4}\right) \ln \frac{1}{r}-\frac{1}{4}  \left(\ln \frac{e}{r}\right)^{\beta} \left(\ln \frac{1}{r_1}\right)^{1-\beta}.
\end{align*}
Using this, alongside the change of variables $z = \ln(1/r)$, and letting $z_1 = \ln (1/r_1)$, we deduce
\begin{align*}
      \int_{\Delta_3} e^{ \alpha_{N,\beta} \frac{u^{N'}(r)}{ X_1^{\beta}(r)}} r^{N-1} \, dr 
 &\leq c_{N,\beta} \int_{\Delta_3} \exp  \left(\frac{1}{4}\ln \frac{1}{r} - \frac{1}{4}  \left(\ln \frac{e}{r}\right)^{\beta} \left(\ln \frac{1}{r_1}\right)^{1-\beta}\right) \, \frac{dr}{r} \notag \\
    &\le c_{N,\beta} \int^{ a_2 z_1}_{a_3z_1} \exp \left(\frac{1}{4}z - \frac{1}{4}z_1^{1-\beta} \left(1+z\right)^{\beta}\right) \, dz\notag \\
    &\stackrel{\eqref{tilde s def}}{\le} c_{N,\beta} \int^{z_1}_{a_3 z_1} \exp \left(\frac{1}{4}z - \frac{1}{4}z_1^{1-\beta} z^{\beta}\right) \, dz \notag \\
    &=c_{N,\beta}z_1\int_{a_3}^{1} e^{\frac{z_1v}{4} - \frac{z_1v^{\beta}}{4}} \, dv\notag \\
    &=c_{N,\beta}z_1 \int_{a_3}^{1} e^{\frac{z_1v^{\beta}}{4} \left( v^{1-\beta} - 1\right)} \, dv \notag \\
    &\le c_{N,\beta} z_1 \int_{a_3}^{1} e^{\frac{a_3^\beta z_1}{4} \left( v^{1-\beta} - 1\right)} \, dv.
\end{align*}
Finally with the change of variables $w= z_1(v^{1-\beta}-1)$, we conclude that  
\begin{align}\label{bound over Delta3}
\int_{\Delta_3} e^{ \alpha_{N,\beta} \frac{u^{N'}(r)}{ X_1^{\beta}(r)}} r^{N-1} \, dr  \leq c_{N,\beta}.
\end{align}

   \par \emph{\text{Estimate in the region }$\Delta_4$}:
Let $r\in \Delta_4 \subset [r_1,1 ]$. We now use \eqref{moser third estimate eq} by  denoting $t \coloneqq \left(\ln \frac{1}{r}/\ln \frac{1}{r_1}\right)^{1-\beta}. $ This yields
\begin{align*} 
    G_{u,\beta}^{N'}(r) -t &\le \left(t + (2\delta)^{1/N} t^{(N-1)/N}\right)^{N'} -t\\
   &\stackrel{\eqref{elementary inequality}}{\le} t^{N'} -t+ N'2^{N'-1}\left( (2 \delta)^{1/N} t^{N'-1+\frac{1}{N'}} + (2 \delta)^{N'/N} t\right) \\
   &= t\left[t^{N'-1} -1 + N'2^{N'-1}\left((2 \delta)^{1/N} t^{N'-2+\frac{1}{N'}}  + (2 \delta)^{N'/N} \right)\right].
\end{align*}
Since $r \in \Delta_4$ so, $0 \le t \le (N-1)/(2N)$.  Also as $N' > 1$ so we have, $N'+ 1/N' > 2$.  Also, note that $0< \delta < 1/{2d} <1/2$ implies, $(2 \delta)^{N'/N} \le (2 \delta)^{1/N} < d^{-1/N}$. Thus the above estimate yields
\begin{align*}
    G_{u,\beta}^{N'}(r) -t &\le t \left( \left(\frac{N-1}{2N}\right)^{N'-1} -1 + N'2^{N'} d^{-\frac{1}{N}}\right)\\
    &\le -\frac{t}{2N} + t \left( N'2^{N'} d^{-\frac{1}{N}}- \frac{2N-1}{2N} + \frac{1}{2^{\frac{1}{N-1}}}  \right)\\
    &= -\frac{t}{2N} + t \left( N'2^{N'} d^{-\frac{1}{N}}-\frac{\left(2-2^{\frac{N-2}{N-1}}\right)N-1}{2N} \right) \stackrel{\eqref{d exact}}{\leq} -\frac{t}{2N}.
\end{align*}
By the above estimate together with \eqref{y(r)}, and Young's inequality we derive,
\begin{align*}
  \alpha_{N,\beta} \frac{u^{N'}(r)} {X_1^{\beta}(r)}= N G_{u,\beta}^{N'}(r) \left( \ln \frac{1}{r_1} \right)^{1-\beta} \left(\ln \frac{e}{r}\right)^{\beta} &\le \left(N-\frac{1}{2}\right) \left(\ln \frac{e}{r}\right)^{\beta} \left(\ln \frac{1}{r}\right)^{1-\beta}\\
    &= \left(N-\frac{1}{2}\right) \beta + \left(N-\frac{1}{2}\right) \ln \frac{1}{r}.
\end{align*}

Therefore, integrating over $\Delta_4,$ we have
\begin{align}\label{bound over Delta4}
 \int_{\Delta_4} e^{ \alpha_{N,\beta} \frac{u^{N'}(r)}{ X_1^{\beta}(r)}} r^{N-1} \, dr 
&\le c_{N,\beta} \int_{\Delta_4}  e^{\left(N-\frac{1}{2}\right) \ln \frac{1}{r}} r^{N-1} \, dr  \notag \\&\le c_{N,\beta}  \int_{r_1}^1 r^{-1/2} \, dr \le c_{N,\beta}.
\end{align}
Finally, we choose 
\begin{equation}\label{choice of delta0}
\delta_0= \frac{1}{2d},
\end{equation}
where $d$ is given by \eqref{d exact}. With this choice of $\delta_0$, we combine \eqref{bound over Delta1}, \eqref{bound over Delta2}, \eqref{bound over Delta3} and \eqref{bound over Delta4} to establish \eqref{second battle eq}. This completes the proof. \\
\end{proof}
Now we prove the boundedness for $\delta_{0} \le \delta \le1.$

\begin{proposition} \label{third battle X_1}
    Let $\delta_{0} \equiv \delta_{0}(N)$ be given by \eqref{choice of delta0}. Then for all $\delta$ satisfying \eqref{i2} and $\delta_0 \leq \delta <1$, we have
     \begin{align}\label{third battle X_1 eq}
       \omega_{N-1} \int_{0}^{1} e^{ \alpha_{N,\beta} u^{N'}(r) X_1^{-\beta}(r)} r^{N-1} \, dr\le c_{N,\beta},
     \end{align}
     for some constant $c_{N,\beta}>0$ depending only on $N$ and $\beta$.
 \end{proposition}
\begin{proof}
     From \eqref{defi F_u beta} and \eqref{i2} we have,
    \begin{align*}
        \omega_{N-1}(1 - \beta)^{N-1} u^N(r) \left( \ln \frac{1}{r}\right)^{(\beta-1)(N-1)} &\le 1- \delta \le 1 - \delta_{0}.
        \end{align*}
        This implies that 
        \begin{align*}
      \alpha_{N,\beta} u^{N'}(r) \left( \ln \frac{e}{r}\right)^{\beta} &\le N (1-\delta_{0})^{\frac{1}{N-1}} \left(\ln \frac{1}{r}\right)^{1- \beta} \left(\ln \frac{e}{r}\right)^{\beta}.
    \end{align*}
    Therefore, we have
    \begin{align*}
       \int_{0}^{1} e^{ \alpha_{N,\beta} u^{N'}(r) X_1^{-\beta}(r)} r^{N-1} \, dr &\le  \int_{0}^{1} e^{ N (1-\delta_{0})^{\frac{1}{N-1}} \left(\ln \frac{1}{r}\right)^{1- \beta} \left(\ln \frac{e}{r}\right)^{\beta}} \, dr\\
        &=  \int_{0}^{\infty} e^{ N (1-\delta_{0})^{\frac{1}{N-1}} z^{1-\beta}(1+z)^{\beta}} e^{-Nz} \, dz \leq c_{N,\beta},
    \end{align*}
    for some constant $c_{N,\beta}>0$ depending only on $N$ and $\beta$. This complete the proof of \eqref{third battle X_1 eq}. \\
    \end{proof}

\subsection{Super-Critical case (\texorpdfstring{$\alpha > N \omega_{N-1}^{\frac{1}{N-1}}(1-\beta)  $}{})} \label{super-critical}

\begin{proposition} \label{final battle}
    For $\alpha > \alpha_{N,\beta} $, we have,
    \begin{align*}
         \sup_{\left \lbrace u \in W^{1,N}_{0,rad}\left(\mathbb{B}_{N}, w_{1\beta}\right) :\left \lVert \nabla u \right\rVert_{N, w_{1\beta}} \le 1\right\rbrace} \int_{\mathbb{B}_N} e^{\alpha \lvert u \rvert ^{N'} \left(\ln \frac{1}{\lvert x \rvert}\right)^{\beta}} \, dx = \infty.
    \end{align*}
\end{proposition}

\begin{proof}
    For $0<\varepsilon<1$, consider the family of functions defined by \eqref{broken-line}. 
   It can be checked that for any $0<\varepsilon<1$, $\left\lVert \nabla \xi_{\varepsilon} \right\rVert_{N, w_{1\beta}}=1.$
   We now denote $\Tilde{\alpha} \coloneqq N\alpha/\alpha_{N,\beta} > N$ and consider the integral,
    \begin{align*}
         \int_{0}^1 e^{\alpha \lvert \xi_{\varepsilon} \rvert ^{N'} \left(\ln \frac{1}{r}\right)^{\beta}} r^{N-1} \, dr&\ge \int_{0}^{\varepsilon} \exp \left(\frac{\alpha}{\omega_{N-1}^{\frac{1}{N-1}}(1-\beta)} \left(\ln \frac{1}{\varepsilon}\right)^{1-\beta} \left(\ln \frac{1}{r}\right)^{\beta}\right) r^{N-1} \, dr\\
        &\ge \int_{0}^{\varepsilon} e^{\Tilde{\alpha} \ln \frac{1}{\varepsilon} } r^{N-1} \, dr= \frac{1}{N} \varepsilon^{N- \Tilde{\alpha}} \to \infty
    \end{align*}
    as $\varepsilon \to 0$. This completes the proof. \\
\end{proof} 
Since $\ln \frac{1}{\lvert x \rvert} \le \ln 
\frac{e}{\lvert x \rvert}$, so we have the following Corollary.
\begin{corollary} \label{final battle X_1}
     For $\alpha > N \omega_{N-1}^{\frac{1}{N-1}}(1-\beta) $, we have,
    \begin{align*}
         \sup_{\left \lbrace u \in W^{1,N}_{0,rad}\left(\mathbb{B}_{N}, w_{1\beta}\right) :\left \lVert \nabla u \right\rVert_{N, w_{1\beta}} \le 1\right\rbrace} e^{\alpha \lvert u \rvert ^{N'} \left(\ln \frac{e}{\lvert x \rvert}\right)^{\beta}} \, dx = \infty.
    \end{align*}
\end{corollary}
   \begin{proof} [Proof of \Cref{moser weighted3}]
       When $\alpha\leq \alpha_{N,\beta}$, the finiteness of  
$\mathcal{C}_{N, \alpha, w_{1\beta},rad}$, as defined in \eqref{equation sharp radial moser beta}, follows from \ref{second battle} and \ref{third battle X_1}. Conversely, \Cref{final battle X_1} proves that 
$\mathcal{C}_{N, \alpha, w_{1\beta},rad}$ is $\infty$ when $\alpha >\alpha_{N,\beta}$.
  This completes the proof. \\ 
\end{proof}

\appendix
\section{The Gap between Leray Energy and weighted Energy}\label{app A}

\begin{lemma} \label{lerray fail}
  Let $\Omega \subset \mathbb{R}^N$ be open bounded containing the origin and $N\ge 3$. Then there does not exist any $c>0$ which is independent of $u$ such that the following holds,
    \begin{align}\label{lerray fail eq}
        I_{N,\Omega}[u] \le c \left\lVert \nabla v \right\rVert_{N,w_{21},\Omega} , \quad \text{for all} \quad u \in C^1_{c}(\Omega \setminus \{0\}),
    \end{align}
    where $v= u X_1^{1-\frac{1}{N}}$ and $I_{N,\Omega}$ is as defined in \eqref{reduced energy leray}.
\end{lemma}
\begin{proof}
    If possible, assume that for some $c>0$, \eqref{lerray fail eq} holds. By \cite[Proposition 2.2]{di2024optimal} we have 
    \begin{align*}
        \int_{\Omega} \lvert x \rvert^{2-N} \lvert v \rvert^{N-2} \lvert \nabla v \rvert^2 X_1^{-1} \, dx \le \kappa_{N} I_N[u]  \quad \text{for all} \quad u \in C^1_{c}(\Omega \setminus \{0\}),
    \end{align*}
    where $\kappa_N =\frac{2}{N}N'^{N-2}$ and $v= u X_1^{1-\frac{1}{N}}$.

Combining this with \eqref{lerray fail eq} and using the fact that $u \in C^1_{c}(\Omega\setminus \{0\})$ we obtain
\begin{align} \label{I_N[u]  J_N[u]}
      \int_{\Omega} \lvert x \rvert^{2-N} \lvert v \rvert^{N-2} \lvert \nabla v \rvert^2 X_1^{-1} \, dx \le c\kappa_{N} \int_{\Omega} \lvert \nabla v\rvert^N w_{21} \, dx,
\end{align}
for all $v \in C^1_c(\Omega \setminus \{0\}).$ Hence by Theorem \ref{approximation} we conclude that, \eqref{I_N[u]  J_N[u]} is true for all $v \in  W^{1,N}_{0}\left(\Omega, w_{21}\right).$ Fix $\delta>0$ such that $B(0,\delta) \subset \Omega$. Now for any $0<r_1<\delta$, consider the family of functions given by,
    \begin{align*}
        v_{r_1}(r) = \begin{cases}
            \left(\ln \ln \frac{e \delta}{r_1}\right)^{1/N'}, \quad 0 \le r < r_1\\
            \frac{\ln \ln \frac{e \delta}{r}}{\left(\ln \ln \frac{e\delta}{r_1}\right)^{1/N}}, \quad r_1 \le r <\delta\\
            0 , \quad r \ge \delta.
        \end{cases}
    \end{align*}
Then by Lemma \ref{contains lip functions}, $v_{r_1} \in W^{1,N}_{0}\left(\Omega, w_{21}\right)$. Also, it is easy to see that $v_{r_1}$ remains bounded in $W^{1,N}_{0}\left(\Omega, w_{21}\right)$ as $r_1\to 0$. On the other hand,
\begin{align*}
\frac{1}{\omega_{N-1}}\int_{\Omega}\frac{ \lvert v \rvert^{N-2}}{ \lvert x \rvert^{2-N}} \left\lvert \nabla v \right\rvert^2 X_1^{-1} \, dx &\geq \int_{r_1}^{\delta} \lvert v_{r_1}(r) \rvert^{N-2} \lvert v_{r_1}'(r)\rvert^2 \ln \frac{e R_{\Omega}}{r} r \, dr\\
    &=\int_{r_1}^{\delta} \frac{\left(\ln \ln \frac{e \delta}{r}\right)^{N-2}}{\ln \ln \frac{e \delta}{r_1}} \frac{1}{\left(\ln \frac{e \delta}{r}\right)^2} \left(\ln \frac{R_{\Omega}}{\delta} + \ln \frac{e \delta}{r}\right)\frac{1}{r} \,dr\\
    &\ge \frac{1}{\ln \ln \frac{e \delta}{r_1}}\int_{0}^{\ln \ln \frac{e \delta}{r_1}} z^{N-2} \, dz\\
    &=\frac{1}{N-1} \left(\ln \ln \frac{e \delta}{r_1}\right)^{N-2} \to \infty 
\end{align*}
as $r_1 \to 0,$ this gives a contraction. Therefore our assumption was wrong and the proof is complete.\\
\end{proof}

\section{Failure of weighted P{\'o}lya--Szeg{\"o} inequality}\label{app B}
An immediate consequence of the following lemma is the failure of the P{\'o}lya--Szeg{\"o} inequality with the weight $w_{i\beta}$. Notably, the same result follows from \eqref{ruf calanchi sup infinite}, \eqref{ruf2} and Theorem \ref{ruf1}.
\begin{lemma} \label{pz fails}
    Let $N\geq 2$, $0<\beta <1$, if $i=1$ and $0<\beta <\infty$, if  $i=2$.  Then there exists a bounded sequence $u_k \in W^{1,N}_{0}(\mathbb{B}_N,w_{i\beta})$ such that,  $\left \lVert \nabla u_k^\ast \right\rVert_{N, w_{i\beta}} \to \infty$, as $k\to \infty$. Here, $u_k^\ast$ denotes the symmetric decreasing rearrangement of $u_k$ with respect to the Lebesgue measure.
 \end{lemma}
\begin{proof} 
   We present the proof for $w_{2 \beta},$ the case $w_{1 \beta}$ is similar. Consider the sequence of functions,
    \begin{align}
        u_{k,a}(x) \coloneqq \begin{cases}
            \frac{1}{\omega_{N-1}^{1/N}} \left(\ln k \right)^{1/N'}, \quad \lvert x - x_p \rvert < \frac{p}{k}\\
            \frac{1}{\omega_{N-1}^{1/N}} \frac{\ln \frac{p}{\lvert x - x_p \rvert}}{\left(\ln k\right)^{1/N}}, \quad \frac{p}{k} \le \lvert x - x_p \rvert < p\\
            0, \quad \text{otherwise},
        \end{cases}
    \end{align}
 where $x_p=(1-p,0,0,...0) \in \mathbb{R}^N$ for some $p \in (0, 1/4)$. Then it is easy to see that $\left \lVert \nabla u_{k,a} \right\rVert_{N, w_{2\beta}} \le c$, for some constant $c\equiv c(p,N,\beta)$.  Now note that,
 \begin{align}
     u^*_{k,a}(x) \coloneqq \begin{cases}
         \frac{1}{\omega_{N-1}^{1/N}} \left(\ln k\right)^{1/N'}, \quad \lvert x \rvert < \frac{p}{k}\\
         \frac{1}{\omega_{N-1}^{1/N}} \frac{\ln \frac{p}{\lvert x \rvert}}{\left(\ln k \right)^{1/N}}, \quad \frac{p}{k} \le \lvert x \rvert < p\\
            0, \quad \text{otherwise}.
     \end{cases}
 \end{align}
 Therefore we have,
 \begin{align*}
     \int_{\mathbb{B}_N} \lvert \nabla u^*_{k,a}(x) \rvert^N \left(\ln \frac{e}{\lvert x \rvert}\right)^{\beta(N-1)} dx &= \frac{\omega_{N-1}}{\ln k} \int_{\frac{p}{k}}^{p}  \left(\ln \frac{e}{r}\right)^{\beta(N-1)} \frac{1}{r}\, dr\\
     &=  \frac{\omega_{N-1}}{\ln k} \frac{\left(\ln \frac{ek}{p}\right)^{\beta(N-1)+1} - \left(\ln \frac{e}{p}\right)^{\beta(N-1)+1}}{\beta(N-1)+1}\\ &\to \infty
 \end{align*}
 as $r_1 \to 0$, since $\beta > 0.$ This completes the proof. \\
 \end{proof}

\section*{Acknowledgement}
\par \emph{Adimurthi} gratefully acknowledges the Indian Institute of Science for the Satish Dhawan Visiting Chair Professorship, which he held during the preparation of this manuscript.
\par \emph{Arka Mallick} is partially supported by ANRF ARG Grant. Grant Number: $ANRF/ARG/2025/000348/MS$.

\end{document}